\documentclass[11pt, a4paper]{article}
\usepackage[DIV=11]{typearea}

\usepackage[utf8]{inputenc}
\usepackage{xcolor}
\usepackage{amssymb}
\usepackage{amsmath}
\usepackage{amsthm}
\usepackage{listings}
\usepackage{dsfont}
\usepackage{mathtools}
\usepackage{enumitem}
\usepackage{bigints}
\usepackage{nicefrac}

\usepackage[language=english, backend=bibtex, style=alphabetic, hyperref, maxbibnames=999, maxcitenames=6, doi=false, isbn=false, url=false]{biblatex}
\AtEveryBibitem{\clearfield{month}}
\AtEveryCitekey{\clearfield{month}}

\usepackage{setspace}
\usepackage[symbol]{footmisc}

\renewcommand{\epsilon}{\varepsilon}
\renewcommand*{\Re}{\ensuremath{\mathrm{Re}~}}
\renewcommand*{\Im}{\ensuremath{\mathrm{Im}~}}
\renewcommand*{\tilde}{\widetilde}
\newcommand{\dd}{\ensuremath{\;\mathrm{d}}}
\newcommand{\loc}{\ensuremath{\mathrm{loc}}}
\newcommand{\tx}{\ensuremath{\underline{x}}}

\newcommand{\JJ}{\ensuremath{\mathcal{J}}}
\newcommand{\J}{\ensuremath{\mathcal{J}_{\R^{d-1}}}}

\DeclareMathOperator{\supp}{supp}

\DeclareMathOperator{\divv}{div}
\DeclareMathOperator{\curl}{curl}

\newcommand\ii{\ensuremath{\mathrm i}}
\newcommand\N{\ensuremath{\mathbb N}}
\newcommand\Q{\ensuremath{\mathbb Q}}
\newcommand\Z{\ensuremath{\mathbb Z}}
\newcommand{\R}{\ensuremath{\mathbb R}}
\newcommand{\C}{\ensuremath{\mathbb C}}
\newcommand{\CC}{\ensuremath{\mathbb C}}
\renewcommand*{\hat}{\widehat}

\newtheorem{theorem}{Theorem}
\newtheorem{lemma}[theorem]{Lemma}

\newtheorem{assumption}{Assumption}
\newtheorem{problem}{Problem}
\newtheorem{remark}[theorem]{Remark}
\newtheorem{proposition}[theorem]{Proposition}

\hypersetup{
    colorlinks=true,
    linkcolor=black,
    citecolor=black,
    filecolor=black,
    urlcolor=black,
    pdftitle={Electromagnetic wave scattering from local perturbed periodic inhomogeneous layers},
    pdfauthor={Alexander Konschin}
}

\title{Electromagnetic wave scattering from locally perturbed periodic inhomogeneous layers}
\author{
    Alexander Konschin
    \footnotemark[1]
}

\date{\today} 
\begin{document}
    \maketitle
    \footnotetext[1]{Center for Industrial Mathematics, University of Bremen, Germany; \texttt{alexk@uni-bremen.de}}

    \begin{abstract}
        We consider the scattering problem on locally perturbed periodic penetrable dielectric layers, which is formulated in terms of the full vector-valued time-harmonic Maxwell's equations. The right-hand side is not assumed to be periodic.
        At first, we derive a variational formulation for the electromagnetic scattering problem in a suitable Sobolev space on an unbounded domain and reformulate the problem into a family of bounded domain problems using the Bloch-Floquet transform. For this family we can show the unique existence of the solution by applying a carefully designed Helmholtz decomposition. Afterwards, we split the differential operator into a coercive part and a compact perturbation and apply the Fredholm theory. Having that, the Sherman-Morrison-Woodbury formula allows to construct the solution of the whole problem handling the singularities of the Calderon operator on the boundary. 
        Moreover, we show some regularity results of the Bloch-Floquet transformed solution w.r.t.\ the quasi-periodicity.
    \end{abstract}
     
    \section{Introduction}
        This study of the electromagnetic scattering problem for locally perturbed periodic layers modeled by the full vector-valued Maxwell's equations is motivated by the growing industrial interest for nano-structured materials and the resulting challenge to construct an automated non-destructive testing method for these kinds of structures.
        Assuming a periodic layer and a quasi-periodic incident field, the problem can be reduced to one periodic cell, which is analyzed in \cite{Bao2000,Dobso1994,Schmi2003} in the case of a constant permeability and in \cite{AmmariBao2003} for a chiral periodic media. 
        The acoustic scattering problem from unbounded periodic structures with a quasi-periodic incident field is a well-established topic in mathematics and was analyzed in various articles (see, e.g., \cite{Abboud1992,Bao1994,Bao1995,BaoDobsomCox95,Bonne1994,Dobson1992,Kirsc1993,Kirsc1995a}).
        If the incident field or the media do not satisfy the periodicity condition, then the problem is usually treated as a rough layer scattering problem where some restrictive assumptions for the permittivity and permeability are prescribed in the literature (see \cite{Hadda2011,LiZhengZheng2016} for the vector-valued problem and \cite{HuLiuQuZhang2015,Lechl2009} for the acoustic scattering problem). Alternatively, the wave number is assumed to be complex valued (see, e.g.,  \cite{LiWuZheng2011}) in which case the sesquilinear form is coercive and the problem is much easier. For sound-soft rough surfaces, the unique existence of the solution is well-known (see \cite{Chand2010} and \cite{Chand2005}). 
        
        In this article, we consider the rough scattering problem in the case that the permittivity $\varepsilon$ and that the permeability $\mu$ are periodic functions in the first two components and the permittivity includes a local defect. The right-hand side can be chosen arbitrarily without any periodicity restrictions. 
        Our Bloch-Floquet transform approach, which is motivated by the results of \cite{Konschin19a} for the Helmholtz equation, allows to show the unique existence of the solution to the vector-valued scattering problem with much less restrictive regularity assumptions on the parameter.  It is enough to assume that the permittivity $\varepsilon$ and the permeability $\mu$ are Lipschitz-continuous functions and some condition for the uniqueness, e.g., that the set $\{\Im \varepsilon > 0\}$ should contain an open ball to avoid surface waves. 
        
        The Bloch-Floquet approach was used in existing work to analyze the acoustic scattering problem in waveguides (compare \cite{Ehrhardt2009,Fliss2015,Joly2006}) and in open space (see \cite{Konschin19a}).
        In \cite{LechleiterZhang2017b} a first approach for the vector-valued problem for a periodic permittivity, a constant permeability and for a non-periodic right-hand side was studied. Because of the choice of a constant permeability the solution to the corresponding variational problem of the Maxwell equations is $H^1$-regular and the boundary condition is well-behaved. In our case we  consider the permeability as a function and
        seek  a solution in the $H(\curl)$ space. In this case the radiation condition has to be adjusted and includes singularities in the frequency domain. This makes the analysis much more involved.

        For the existence theory, we  apply the Bloch-Floquet transform to derive a family of quasi-periodic variational problems in a bounded domain. As the next step, we  divide the quasi-periodic solution space into the sum of three subspaces by constructing two suitable Helmholtz decompositions, such that the problem reduces to finding a solution in a more regular subspace.
        In this quasi-periodic  subspace, we can split the sesquilinear form into a coercive part and a compact perturbation and apply the Fredholm theory to conclude the solvability of the family of problems by showing the uniqueness. Having this, we  construct a solution to the actual problem by showing and applying the so-called Sherman-Morrison-Woodbury formula to handle the singularities.
        Furthermore, the formula allows to show the regularity of the Bloch-Floquet transformed solution with respect to the quasi-periodicity and describes a natural way to approximate the solution numerically.
        
        This paper is organized as follows. In section \ref{Sec2} we  introduce the scattering problem and in section \ref{Sec3} we derive the corresponding variational formulation using the Calderon operator for the upper boundary condition. Hereafter, we prove the unique existence of the solution in section \ref{Sec4}, if the permittivity is not perturbed, and in section \ref{Sec5} we prove the unique existence of the solution for the locally perturbed permittivity. Section \ref{Sec6} contains regularity results of the Bloch-Floquet transformed solution.

    \section{Scattering problem}\label{Sec2}
        We model the scattered electric field ${\mathbf E}$ as well as the magnetic field ${\mathbf H}$   as the solutions to the Maxwell's equations in $\R^3$.
        We assume to have an inhomogeneous isotropic material, or in other words, we assume that the permeability $\mu$, the permittivity $\varepsilon$ and the resistance $\sigma$ are scalar-valued functions in $L^\infty(\R^3_+, \R)$ and fulfill the bounds $\mu(x)  \geq  \mu_0 > 0$, $\varepsilon(x)  \geq  \varepsilon_0 > 0$ and $\sigma(x)  \geq 0$ for all $x \in \R^3_+ := \R^2 \times (0, \infty)$. 
        We consider the problem with a perfect conductor on the lower boundary $\R^2 \times \{ 0 \}$ for better readability, which implicates the boundary condition ${\mathbf E}_T := ({\mathbf E}_1, {\mathbf E}_2, 0)^\top = 0$ on $\R^2 \times \{ 0 \}$. The arguments can be easily extended to the case that the obstacle is  surrounded by homogeneous media.

		Define for $R \geq 0$ the sets
		\begin{align*}
			\Omega^R &:= \R^{2} \times (0, R)
			,
			&\Gamma^R &:= \R^{2} \times \{R\}
			,
			\\
			\Gamma_0^R &:= (-\pi, \pi)^{2} \times \{R\}
			\text{ and }
			&I_{} &:= (\nicefrac{-1}{2}, \nicefrac{1}{2})^{2}
			.
		\end{align*}
        We seek  the electric field ${\mathbf E}$ as well as the magnetic field ${\mathbf H}$ as function of the space $H_{\loc}(\curl; \R^3_+)$, where both lie in $H(\curl; \Omega^R)$ for all  $R \geq R_0 > 0$ and fulfill the equations
        \begin{subequations}\label{eq_Maxwell1}
            \begin{align}
            \nabla \times {\mathbf E} - \ii \omega \mu {\mathbf H} 
            &= 0
            ,&
            &\nabla \times {\mathbf H} + \ii (\omega \varepsilon + \ii \sigma) {\mathbf E} 
            = J
            &&\text{in }
            L^2(\R^3_+)^3
            ,
            \\
            \divv(\mu {\mathbf H})
            &= 0
            ,&
            &\divv \left( \varepsilon {\mathbf E} +\ii \frac{\sigma}{\omega} {\mathbf E} \right)
            = \frac{1}{\ii \omega}\divv ( J)
            &&\text{in }
            H^{-1}(\R^3_+)
            ,
            \\
            {\mathbf E}_T\big|_{\Gamma^0 } 
            &= 0
            &
            &
            &&\text{in }
            L^2(\R^2)
            ,
            \end{align}
        \end{subequations}
        where  $\omega \in \R_{\geq 0}$, $\mu$, $\varepsilon$, $\sigma \in L^\infty(\R^3_+, \R)$ and $J \in L^2(\R^3_+)^3$.
        Since the electric field is a function in $H(\curl; \Omega^R)$ for some $R > R_0$, the trace ${\mathbf E}_T\big|_{\Gamma^0 } $ is well-defined on $\Gamma^0$.
        The substitution of the magnetic field ${\mathbf H}$ gives the Maxwell's equation of second order for the electric field ${\mathbf E} \in H(\curl; \Omega^R)$ for all $R \geq R_0 > 0$ of the form
        \begin{subequations}\label{eq_Maxwell3_1}
            \begin{align}
            \nabla \times \biggl( \mu^{-1} \nabla \times {\mathbf E} \biggr) - \omega^2 \epsilon {\mathbf E} 
            &= f
            := \ii \omega  \mu_+ J
            &&\text{in }
            L^2(\R^3_+)^3
            ,
            \label{eq_Maxwell3_a}
            \\
            \divv (\epsilon  {\mathbf E})
            = -\frac{1}{\omega^2 }  \divv  f
            &= \frac{1}{\ii \omega }\divv  J
            &&\text{in }
            H^{-1}(\R^3_+)
            ,
            \\
            {\mathbf E}_T\big|_{\Gamma^0 } 
            &= 0
            &&\text{in }
            L^2(\R^2)
            .
            \end{align}
        \end{subequations}
        Having a solution ${\mathbf E}$ to the equations~\eqref{eq_Maxwell3_1},  the functions ${\mathbf E}$ and ${\mathbf H} := \tfrac{1}{\ii \omega \mu_r \mu_+} \nabla \times {\mathbf E}$ solve the first order Maxwell's equations~\eqref{eq_Maxwell1}.
        
        We assume that the obstacle is bi-periodic in the first two components  $\tx := (x_1, x_2) \in \R^2$ and bounded in the third direction. We characterize the periodicity  by some invertible matrix $\Lambda \in \R^{2 \times 2}$ and  set $\Lambda^*:= 2 \pi (\Lambda^{T})^{-1}$. With the boundedness of the object in the third direction we describe the fact that we can find an $R_0 > 0$, such that the obstacle is supported in the strip $\Omega^{R_0}$. 
        To simplify the notation, we assume without loss of generality that the periodicity equals to the scaled identity matrix $\Lambda = 2\pi I_{2} \in \R^{2 \times 2}$, $\Lambda^* = I_2$ and that the local perturbation $q \in L^\infty(\R_+^{3})$ has the support in 
		$\Omega_0^{R_0}$, where $\Omega_0^{R} := (-\pi, \pi)^{2} \times (0, R)$ 
		for $R \geq R_0$.
        For the existence theory we prescribe the following assumptions:
        \begin{itemize}
            \item 
            The constant $R_0 > 0$ is chosen, such that for a small $\delta > 0$ the parameter $ \varepsilon +\tfrac{\ii\sigma}{\omega}$ and $\mu$ are constant outside of $\Omega^{R_0-\delta}$ with $\sigma = 0$. 
            In other words, the parameter can be described by constants $\mu = \mu_+ > 0$ and $ \varepsilon +\tfrac{\ii\sigma}{\omega} = \varepsilon_+ >0$ in $\R^3_+ \setminus \overline{\Omega^{R_0-\delta}}$. 
            We set the abbreviation $\varepsilon_r :=  ( \varepsilon +  \tfrac{\ii\sigma}{\omega}) /\varepsilon_+$ and $\mu_r :=  \tfrac{\mu}{\mu_+}$ as well as $k^2:=\omega^2 \mu_+\varepsilon_+$.
            
            \item 
            The permittivity $\varepsilon_r$ and the permeability $\mu_r$ are  $2\pi$-periodic and we write  ${\varepsilon}^{\mathrm{s}}_r := \varepsilon_r +q$ for the perturbed permittivity with the perturbation $q \in L^\infty(\R^3_+)$, which is supported in $\Omega_0^{R_0}$.
            
            \item The right-hand side $f$  is supported in $\Omega^{R_0-\delta}$.
        \end{itemize}
        The assumptions allow us to derive a radiation condition for the scattering problem.
        For a sufficient small $\delta > 0$ the parameter $\mu_r$ and $\epsilon_r$ are constant in $\R^3_+ \setminus \overline{\Omega^{R_0-\delta}}$ and the right-hand side vanishes there. Hence, for the divergence of ${\mathbf E}$ it holds $\divv {\mathbf E} = 0$ above the strip domain and the equation~\eqref{eq_Maxwell3_a} reduces to the component-wise homogeneous Helmholtz equation $\Delta {\mathbf E} + k^2 {\mathbf E} = 0$.
        In consequence, the scattered field ${\mathbf E}$ should satisfy the angular spectrum representation as the radiation condition, which is defined by ${\mathbf E}^{R}\in C^\infty \left(\R^2 \times (R, \infty) \right)$,
        \begin{equation}\label{eq_Ausstrahl1}
            {\mathbf E}^{R}(x)
            := \frac{1}{{2\pi}} \int_{\R^2} \mathcal{F} \left( {\mathbf E}\big|_{ \Gamma^{R}} \right) (\xi ) \,e^{\ii \xi \cdot \tx + \ii \sqrt{k^2 - |\xi|^2}(x_3 - R)} \dd  \xi
        \end{equation}
        for $x \in \R^3_{+} \setminus \overline{\Omega^{R_0-\delta}}$.
        We call the space of functions $\phi_T = (\phi_1, \phi_2, 0)^T \in H^{\nicefrac{1}{2}}(\Gamma^{R})^3$ as ${TH}^{\nicefrac{1}{2}} (\Gamma^{R})$. The following regularity theorem shows that the radiation condition~\eqref{eq_Ausstrahl1} is well-defined.
        \begin{theorem}\label{satz_regGebiet}
            Choose $R' > R \in \R$ and set  $\Omega := \R^2 \times (R, R')$. If a function $u \in H(\curl; \Omega ) \cap H(\divv; \Omega )$ satisfies  $u_T \in TH^{\nicefrac{1}{2}}(\partial \Omega)$, then $u \in H^1(\Omega )^3$ and it holds the estimation
            \begin{align*}
                ||u||_{H^1(\Omega )^3}
                \leq c \biggl( ||u||_{H(\curl; \Omega )} + ||\divv u||_{L^2(\Omega )}
                + || u_T ||_{TH^{\nicefrac{1}{2}}(\partial \Omega)} \biggr)
                .
            \end{align*}
        \end{theorem}
        \begin{proof}
            For a smooth function $\phi \in C^\infty_0(\R^3)^3$ with compact support in $\R^3$ it holds $\nabla \times \nabla \times \phi = \nabla \divv \phi - \Delta \phi$ and Green formula gives
            \begin{align*}
                &\int_{\Omega} |\nabla \phi |^2 \dd x
                = -\sum_{j=1}^{3} \int_{\Omega} \Delta \phi_j\, \overline{\phi}_j \dd x
                + \left[ \int_{\R^2 \times \{R'\}} - \int_{\R^2 \times \{R\}} \right] \frac{\partial}{\partial x_3} \phi_j\, \overline{\phi}_j \dd S
                \\
                &= \int_{\Omega} \nabla \times \nabla \times \phi \cdot \overline{\phi} - (\nabla \divv \phi) \cdot \overline{\phi} \dd x
                + \sum_{j=1}^{3}
                \left[ \int_{\R^2 \times \{R'\}} - \int_{\R^2 \times \{R\}} \right] \frac{\partial \phi_j}{\partial x_3}   \overline{\phi}_j  \dd S
                .
            \end{align*}
            If we apply the Green formula a second time, we derive
            \begin{align*}
                \int_{\Omega} |\nabla \phi |^2 \dd x
                &
                = \int_{\Omega} \nabla \times \phi \cdot \nabla \times \overline{\phi} + \divv \phi\, \divv\overline{\phi} \dd x
                \\
                &\quad
                + \sum_{j=1}^{3}
                \left[ \int_{\R^2 \times \{R'\}} - \int_{\R^2 \times \{R\}} \right] \frac{\partial}{\partial x_3} \phi_j\, \overline{\phi}_j \dd S
                \\
                &\quad 
                + \left[ \int_{\R^2 \times \{R'\}} - \int_{\R^2 \times \{R\}} \right] e_3 \times (\nabla \times \phi) \cdot \overline{\phi} - \divv \phi\,\overline{\phi}_3  \dd S
                .
            \end{align*}
            For the boundary term on $\Gamma :=\R^2 \times \{R'\}$, and analogously on $\R^2 \times \{R\}$, we compute 
            \begin{align*}
                & \int_{\Gamma} e_3 \times (\nabla \times \phi) - \divv \phi\,\overline{\phi}_3 
                + \sum_{j=1}^{3}
                \frac{\partial}{\partial x_3} \phi_j\, \overline{\phi}_j \dd S
                \\
                &=   
                \int_{\Gamma}  \frac{\partial}{\partial x_2} \phi_3\, \overline{\phi}_2 
                + \frac{\partial}{\partial x_1} \phi_3\, \overline{\phi}_1 
                - \frac{\partial}{\partial x_1 } \phi_1\, \overline{\phi}_3 
                - \frac{\partial}{\partial x_2 } \phi_2\, \overline{\phi}_3 
                \dd S
                \\
                &=  - 2\, \Re \int_{\Gamma} (\divv_T \phi_T )\, \overline{\phi}_3  \dd S
                .
            \end{align*}
            Since the operator $\divv_T \colon H^{\nicefrac{1}{2}}(\R^2)^2 \to H^{\nicefrac{-1}{2}}(\R^2)$ is continuous, we conclude
            \begin{align*}
                ||\phi||^2_{H^1(\Omega)^3}
                &\leq ||\phi||^2_{H(\curl; \Omega)} + ||\divv \phi||^2_{L^2(\Omega)}
                + C ||\phi_T||_{TH^{\nicefrac{1}{2}}(\partial \Omega)} ||\phi||_{H^1(\Omega)^3}
                \\
                &\leq \left(||\phi||_{H(\curl; \Omega)} + ||\divv \phi||_{L^2(\Omega)}
                + C ||\phi_T||_{TH^{\nicefrac{1}{2}}(\partial \Omega)} \right)
                ||\phi||_{H^1(\Omega)^3}
            \end{align*}
            for all  $\phi \in C_0^\infty(\R^3)^3$.
            Because  of the denseness of the space $C_0^\infty(\R^3)^3$ of smooth functions with compact support in \mbox{$\{ u \in H(\curl; \Omega ) \cap H(\divv; \Omega ) : u_T \in TH^{\nicefrac{1}{2}}(\partial \Omega) \}$}, the estimation holds for every function $u \in H(\curl; \Omega ) \cap H(\divv; \Omega )$ with $u_T \in TH^{\nicefrac{1}{2}}(\partial \Omega)$.
        \end{proof}
        We  seek a solution to the Maxwell's equations which has the trace in $TH^{\nicefrac{1}{2}}(\Gamma^R)$. Therefore, the solution is locally an $H^1$-regular  around the boundary $\Gamma^R$ and analogously to \cite[Lemma~2.2]{Chand2005} and we can prove that the radiation condition ${\mathbf E}^R$ is well-defined and lies in  $H^1(\R^2 \times (R, R'))$.

    \section{Reduction to a variational problem}\label{Sec3}
        We reduce the scattering problem ~\eqref{eq_Maxwell3_1},~\eqref{eq_Ausstrahl1} to a variational problem. For that, we first apply formally the Gaussian theorem on the domain $\Omega^R$ to derive equation
        \begin{align*}
            \int_{\Omega^R} 
            \nabla \times \biggl( \mu_r^{-1} \nabla \times {\mathbf E} \biggr) \cdot \overline{v} - k^2 {\varepsilon}^{\mathrm{s}}_r {\mathbf E} \cdot \overline{v}
            \dd  x
            &
            =
            \int_{\Omega^R} 
            \mu_r^{-1} \nabla \times {\mathbf E}  \cdot \nabla \times \overline{v} 
            - k^2{\varepsilon}^{\mathrm{s}}_r {\mathbf E} \cdot \overline{v} 
            \dd  x
            \\
            &\quad\quad
            + \left[ \int_{\Gamma^{R}} - \int_{\Gamma^{-R}} \right] 
            \left[ e_3 \times (\nabla \times {\mathbf E}) \cdot \overline{v} \right] 
            \dd  S
            .
        \end{align*}
        In the following we  replace the boundary term by some boundary condition and adapt the radiation condition, such that the solution to the variational problem  solve the scattering problem~\eqref{eq_Maxwell3_1},~\eqref{eq_Ausstrahl1}.
        Since the trace of the third component ${\mathbf E}_{3}\big|_{\Gamma^{R}}$ of a function $H(\curl; \Omega^R)$ is not well-defined, we have to reformulate the radiation condition.
        
        Near the boundary  $\Gamma^{R}$ for some $R \geq R_0$ the $H^1$-regularity of the solution follows by theorem~\ref{satz_regGebiet}. For a sufficient small $ \delta > 0$  it follows by standard regularity results that the solution actually lies in $H^2(\R^2 \times (R-\delta, R+\delta))$ and, in particular, the solution ${\mathbf E}$  satisfies the equation
        \[
            e_3 \times (\nabla \times {\mathbf E})
            = \nabla_T {\mathbf E}_3 - \frac{\partial {\mathbf E}_T}{\partial x_3}
            \quad\text{in }
            H^{\nicefrac{1}{2}}(\Gamma^R)
            .
        \]
        Since for $R \geq R_0$ the solution is given by the radiation condition~\eqref{eq_Ausstrahl1}, 
        we can express $\frac{\partial {\mathbf E}_T}{\partial x_3}\big|_{\Gamma^{R}}$ as 
        \begin{align}\label{eq_T}
            \frac{\partial {\mathbf E}_T}{\partial x_3}\big|_{\Gamma^{R}}
            = T( {\mathbf E}_T\big|_{\Gamma^{R}})(\tx, R)
            := \frac{\ii}{2\pi}  \int_{\R^{2}}   \sqrt{k^2 - | \xi|^2}\, \mathcal{F}( {\mathbf E}_T \big|_{\Gamma^{R}}) (\xi)\, e^{\ii \tx \cdot \xi} \dd \xi
        \end{align}
        for $\tx \in \R^{2}$.
        The operator $T \colon H^{\nicefrac{1}{2}}(\Gamma^{R}) \to H^{\nicefrac{-1}{2}}(\Gamma^{R})$ is called the Dirichlet-to-Neumann operator and satisfies the inequalities
        \[
            -\Re \langle T \phi, \phi \rangle_{H^{\nicefrac{-1}{2}}(\Gamma^{R}) \times H^{\nicefrac{1}{2}}(\Gamma^{R})}
            \geq 0
            \quad
            \text{and}
            \quad
            \Im \langle T \phi, \phi \rangle_{H^{\nicefrac{-1}{2}}(\Gamma^{R}) \times H^{\nicefrac{1}{2}}(\Gamma^{R})}
            \geq 0
        \]
        (compare \cite{Chand2005}).
        For the remaining term $\nabla_T {\mathbf E}_3$ we consider the identity
        \[
            0
            = \divv {\mathbf E}
            = \divv_T {\mathbf E}_T + \frac{\partial {\mathbf E}_3}{\partial x_3}
            \quad \text{in }
            H^1 \left(\R^2 \times (R-\delta, R+\delta)\right)
        \]
        together with the Fourier-transform and \eqref{eq_T} to derive
        \[
            \ii \xi \,\mathcal{F} \left({\mathbf E}_{3}\big|_{\Gamma^{R}} \right)
            = -\frac{\ii \xi}{\sqrt{k^2 - | \xi|^2}}\, \left(\xi \cdot \mathcal{F}({\mathbf E}_{T}\big|_{\Gamma^{R}})\right)
            \quad \text{in }
            L^2\left(\R^2\right)
            .
        \]
        We define the space $\hat{TH}^{\nicefrac{1}{2}} (\Gamma^{R})$ as 
        \begin{align*}
            \hat{TH}^{\nicefrac{1}{2}} (\Gamma^{R})
            := \left\{ u_T \in {TH}^{\nicefrac{1}{2}} (\Gamma^{R}) : \biggl(\xi \mapsto \frac{\xi \cdot \mathcal{F}(u_T)(\xi)}{|k^2 - |\xi|^2|^{\nicefrac{1}{4}}} \biggr)\in L^2(\R^2) \right\}
            ,
        \end{align*}
        which has the norm
        \begin{equation*}
            ||u_T||^2_{\hat{TH}^{\nicefrac{1}{2}}(\Gamma^{R})}
            := 
            \bigintsss_{\R^2} \frac{1}{|k^2 - |\xi|^2|^{\nicefrac{1}{2}}}\, |\xi \cdot \mathcal{F}(u_T)(\xi)|^2
            +|1 + |\xi|^2|^{\nicefrac{1}{2}}\, |\mathcal{F}(u_T)(\xi)|^2
            \dd \xi
            ,
        \end{equation*}
        and the corresponding scalar product.
        By construction of the space $\hat{TH}^{\nicefrac{1}{2}} (\Gamma^{R})$, the operator $N \colon \hat{TH}^{\nicefrac{1}{2}} (\Gamma^{R}) \to \hat{TH}^{\nicefrac{1}{2}} (\Gamma^{R})'$,
        \[
            N(\phi_T) (\tx)
            := - \frac{1}{2\pi} \int_{\R^2} \frac{\ii \xi}{\sqrt{k^2 - | \xi|^2}}\, \left(\xi \cdot \mathcal{F}(\phi_T)(\xi)\right)\, e^{\ii \xi \cdot \tx} \dd  \xi
            ,
        \]
        is well-defined and satisfies the inequalities
        \[
            -\Re \langle N \phi_T, \phi_T \rangle
            \geq 0
            \ \ \quad\text{and}\quad
            \quad
            -\Im \langle N \phi_T, \phi_T \rangle
            \geq 0
            .
        \]
        Thus, we can define the solution space $X$ as
        \[
            X
            := \left\{ u \in H(\curl; \Omega^R) : u_{T}\big|_{\Gamma^{R}} \in \widehat{TH}^{\nicefrac{1}{2}}(\Gamma^{R}),\ u_{T}\big|_{\Gamma^{0}}=0 \right\}
            ,
        \]
        where the norm is given by
        $
            ||u||^2_X 
            := ||u||^2_{H(\curl; \Omega^R)} 
            + ||u_T||^2_{\hat{TH}^{\nicefrac{1}{2}}(\Gamma^{R})}
            .
        $
        Therefore, we have derived the following variational formulation:
        \begin{problem}\label{prob_Maxwell}
            {We seek ${\mathbf E} \in X$, such that} 
            \begin{align*}
                a_q({\mathbf E}, v)
                &:= \int_{\Omega^R} 
                \mu_r^{-1} \left(\nabla \times {\mathbf E}\right) \cdot \left(\nabla \times \overline{v } \right) 
                - k^2{\varepsilon}^{\mathrm{s}}_r {\mathbf E} \cdot \overline{v} 
                \dd  x
                \nonumber
                \\
                &\quad \quad 
                + \int_{\Gamma^{R}} 
                N({\mathbf E}_{T}\big|_{\Gamma^{R}})  \cdot \overline{v}_{T}\big|_{\Gamma^{R}}- T({\mathbf E}_{T}\big|_{\Gamma^{R}})  \cdot \overline{v}_{T}\big|_{\Gamma^{R}}   
                \dd  S
                \nonumber
                \\ 
                &= \int_{\Omega^R} f \cdot  \overline{v} \dd x
                \end{align*}
            {is satisfied for all   $v \in X$.}
        \end{problem}
        For the existence theory we assume some regularity for the parameter.
        \begin{assumption}\label{assumption_absorptionMaxwell}
            The set $\{\Im \epsilon_r > 0\}$ includes an open ball. Furthermore, the parameter $\epsilon_r$, $q$ and $\mu_r$ should be functions in $W^{1,\infty}(\R^3_+)$ and bounded from below by $\Im {\varepsilon}^{\mathrm{s}}_r \geq 0$, $\Im \mu_r \geq 0$ as well as  by $\Re \epsilon_r \geq \epsilon_0 > 0$ and $\Re \mu_r \geq \mu_0 > 0$.
        \end{assumption}
        
        Having a solution to the variational problem \ref{prob_Maxwell}, we extend the function by the radiation condition and it solves the scattering problem. This results was proven in \cite{Hadda2011}, which we  summarize here.

        \begin{lemma}\label{lemma_rueckUndRand}
            The solution to the variational problem ~\ref{prob_Maxwell} is a distributional solution to the equations~\eqref{eq_Maxwell3_1} in $\Omega^R$ and the equation ~\eqref{eq_Maxwell3_a} holds in $L^2(\Omega^R)$. Moreover, the identity 
            \begin{equation}\label{eq_CalderonRand}
                e_3 \times (\nabla \times {\mathbf E})\big|_{ \Gamma^{R} }
                = N({\mathbf E}_{T}\big|_{\Gamma^{R}}) -  T({\mathbf E}_{T}\big|_{\Gamma^{R}})
            \end{equation}
            holds in $H^{\nicefrac{-1}{2}}(\Gamma^{R})^3$.
        \end{lemma}
        \begin{proof}
            We refer to \cite[Korollar~3.2]{Hadda2011}.
        \end{proof}
        
        \begin{lemma}\label{lemma_ausstrahlungsbedingung}
            The extension of the solution ${\mathbf E}\in X$ to the variational problem~\ref{prob_Maxwell}  for some $R \geq R_0$ to ${\mathbf E}'$ defined by ${\mathbf E}' := {\mathbf E}$ in $\Omega^R$  and
            \begin{subequations}\label{eq_Ausstrahl2_1}
                \begin{align}
                {\mathbf E}'_T(x) &:= \int_{\R^2} \mathcal{F}({\mathbf E}_T)(\xi)\, e^{\ii \xi \cdot \tx + \ii \sqrt{k^2-|\xi|^2}(x_3 - R)}   \;d \xi
                \\
                {\mathbf E}'_3(x) &:=  \int_{\R^2}  \frac{- 1}{\sqrt{k^2 - | \xi|^2}}\, \left(\xi \cdot \mathcal{F}({\mathbf E}_T)(\xi) \right)\, e^{\ii \xi \cdot \tx + \ii \sqrt{k^2-|\xi|^2}(x_3 -  R)} \;d \xi
                \end{align}
            \end{subequations}
            for $x \in \R^3_+\setminus\overline{\Omega^R}$ solves the scattering problem~\eqref{eq_Maxwell3_1},~\eqref{eq_Ausstrahl1} in $\R^3_+$. 
            Furthermore, the extension ${\mathbf E}'$ solves the variational problem~\ref{prob_Maxwell} for $R' > R$ and it holds ${\mathbf E}' \in H^1(\Omega^{R'} \setminus\Omega^{R_0 - \delta})^3$ for all $R' \geq R_0$ and for  $\delta > 0$ small enough.
        \end{lemma}
        \begin{proof}
            We refer to \cite[Korollar~3.3]{Hadda2011}.
        \end{proof}
        
        Having the $H^1$-regularity of the solution ${\mathbf E} \in X$ to the variational problem near the boundary $\Gamma^{R}$, we can conclude that the third component is well-defined and can be characterized by
        \[
            \mathcal{F}({\mathbf E}_3)
            = \frac{- 1}{\sqrt{k^2 - | \xi|^2}}\, \left(\xi \cdot \mathcal{F}({\mathbf E}_T)(\xi) \right)
            \quad \text{in }
            H^{\nicefrac{1}{2}}(\Gamma^{R})
        .
        \]
        Thus, the solution satisfies the radiation condition~\eqref{eq_Ausstrahl1}.

    \section{Existence theory for the periodic permittivity}\label{Sec4}
        At first, we consider the case that both parameter are periodic and that there is  no perturbation  in the permittivity, or in other words that $q=0$ and $\varepsilon_r^\mathrm{s} = \varepsilon_r$. This allows us to apply the Bloch-Floquet transform and consider the quasi-periodic problem. For the quasi-periodic problem we  decompose the solution space with a carefully chosen Helmholtz decomposition to gain a reduced problem on a more regular solution space, which is compactly embedded in $L^2(\Omega_0^R)$. Afterwards, we have to construct the solution to problem \ref{prob_Maxwell} by analyzing the behavior of the quasi-periodic solution operator w.r.t.\ the quasi-periodicity.

        A function is called $\alpha$-quasi-periodic with $\alpha \in \R^{2}$ and period $2 \pi$, if
        \[
            u(\tx + 2 \pi j, x_d)
            = e^{-2 \pi \ii \alpha \cdot j} u(\tx, x_d)
            \quad\text{holds for all }
            j \in \Z^{2}
            .
        \]
        For smooth functions $\phi \in C^\infty_0(\overline{\Omega^R})$ the horizontal Bloch-Floquet transform $\JJ_{\R^2}$ is defined by
        \[
            \JJ_{\R^2} \phi (\alpha, \tx, x_d)
            :=  \sum_{j \in \Z^{2}} \phi (\tx + 2 \pi j, x_d) e^{2 \pi \ii \alpha \cdot j}
            .
        \]
        Recall the spaces $H^s_\alpha(\Omega_0^R)$ and $H^s_\alpha(\Gamma_0^R)$ of $\alpha$-quasi-periodic weakly differentiable functions, and set $\tilde{H}^s_\alpha(\Omega_0^R)$ as the subspace of functions $u \in H^s_\alpha(\Omega_0^R)$ satisfying $u \big|_{\Gamma_0^0} = 0$. The Bloch-Floquet transform extends for $s \in \R$ to an isomorphism between $\tilde{H}^s(\Omega^R)$ and $L^2(I_{}; \tilde{H}^s_\alpha(\Omega_0^R))$ as well as between ${H}^s(\Gamma^R)$ and $L^2(I_{}; {H}^s_\alpha(\Gamma_0^R))$, where the index $\alpha$ indicates that the space depends on $\alpha \in I_{}$ (see \cite{Lechleiter2016}).
        The inverse of the transform is given by
        \[
            \JJ_{\R^2}^{-1} w
            (\tx + 2 \pi j, x_d)
            =  \int_{I_{}} w(\alpha, \tx, x_d) e^{-2 \pi \ii \alpha \cdot j}
            \dd \alpha
            \quad
            \text{for } 
            x \in \Omega_0^R
            ,
            \ 
            j \in \Z^{2}
            .
        \]

        \subsection{Alternative formulation on a bounded domain}
            At first, we  formulate the quasi-periodic scattering problem, which is not well-defined for some quasi-periodicities. For that, we define the set of singularities as
            \[
                \mathcal{A}
                := \left\{
                \alpha \in \overline{I_{}} : |\alpha + j| = k
                \quad\text{for some }
                j \in \Z^2
                \right\}
                .
            \]
            For this problem we consider functions, which are quasi-periodic in $\tx$. Hence, the radiation condition reduces to the Rayleigh radiation condition and we adapt the boundary condition. We write $u_\alpha:=\tilde{u}(\alpha, \cdot)$ for $\alpha \in \overline{I}_{}$ and $\tilde{u} \in L^2(I_{}; L^2(\Omega_0^R))^3$ and define for 
            $\alpha \in \overline{I} \setminus \mathcal{A}$ the space $\tilde{X}_\alpha$ as
            \begin{align*}
                \tilde{X}_\alpha
                &:= \left\{ {\mathbf E}_\alpha \in H_\alpha(\curl; \Omega_0^R) : 
                {\mathbf E}_{\alpha,T} \big|_{\Gamma_0^R} \in TH_\alpha^{\nicefrac{1}{2}}(\Gamma_0^{R})
                ,\ 
                {\mathbf E}_{\alpha,T} \big|_{\Gamma_0^0} =0
                \right\}
                ,\ 
            \end{align*}
            where $H_\alpha(\curl; \Omega_0^R)$ is the subspace of $L^2(\Omega^R_0)$-functions which are $\alpha$-quasi-periodic and which the $\curl$ operator maps into $L^2(\Omega^R_0)$. The trace of these functions can be written as a Fourier series. Since we only need the transverse part of a vector field on the boundary, we  write  ${\mathbf E}_{\alpha, T}$ instead of ${\mathbf E}_{\alpha, T}\big|_{\Gamma_0^R}$ from now on.
            
            Analogously to the continuous problem, we avoid the trace of the third component in the sesquilinear form, since it is not well-defined for all $H_\alpha(\curl; \Omega_0^R)$-functions. Thus, we derive the extension 
            \begin{equation*}
                {\mathbf E}_{\alpha, 3}^{R}(\tx, x_3)
                = \frac{1}{2 \pi}   \sum_{j \in \Z^2} \frac{ 1}{\beta_j} \left( \alpha_j \cdot \widehat{({\mathbf E}_{\alpha, T})}_j \right)\, e^{-\ii \alpha_j \cdot \tx+  \ii\beta_j(x_3 - R)}
            \end{equation*}
            for $x \in \R^2 \times (R, \infty)$.
            For all $\alpha \in \overline{I}$ we define the quasi-periodic Dirichlet-to-Neumann operator $T_\alpha\colon TH_\alpha^{\nicefrac{1}{2}}(\Gamma_0^R) \to TH_\alpha^{\nicefrac{1}{2}}(\Gamma_0^R)'$ for $\phi_T = \sum_{j \in \Z^{2}} \hat{(\phi_T)}_j(\alpha) e^{ -\ii (\alpha + j) \cdot \tx}(\tx) $ by
            \[
                T_\alpha(\phi_T)(\tx)
                = 
                \sum_{j \in \Z^{2}} \sqrt{k^2 - |\alpha + j |^2} \hat{(\phi_T)}_j(\alpha) e^{-\ii (\alpha + j) \cdot \tx}
                ,
            \]
            which satisfies the inequalities
            \begin{equation}\label{eq_Tinequalities}
                \Re \langle T_\alpha \phi_T, \phi_T \rangle
                \leq 0
                \ \ \quad\text{and}\quad
                \quad
                \Im \langle T_\alpha \phi_T, \phi_T \rangle
                \geq 0
                .
            \end{equation}
            For all $\alpha \in \overline{I}\setminus \mathcal{A}$ we define the operator $N_\alpha \colon TH_\alpha^{\nicefrac{1}{2}} (\Gamma_0^{R}) \to TH_\alpha^{\nicefrac{1}{2}} (\Gamma_0^{R})'$ as 
            \begin{equation*}
                N_\alpha(\phi_T ) (\tx)
                :=  - \frac{1}{2 \pi}  \sum_{j \in \Z^2} \frac{\ii\alpha_j}{\sqrt{k^2 - | \alpha_j|^2}} \left(\alpha_j \cdot \widehat{(\phi_T )}_j \right) e^{-\ii \alpha_j \cdot \tx}
                ,
            \end{equation*}
            which is well-defined and satisfies the inequalities
            \begin{equation}\label{eq_Ninequalities}
                -\Re \langle N_\alpha \phi_T, \phi_T \rangle
                \geq 0
                \ \ \quad\text{and}\quad
                \quad
                -\Im \langle N_\alpha \phi_T, \phi_T \rangle
                \geq 0
                .
            \end{equation}
            Moreover, the solution space $\tilde{X} := \JJ_{\R^2}  X = L^2_{\mathrm{w}}(I_{}; \tilde{X}_\alpha)$ is given by 
            \[
                L^2_{\mathrm{w}}(I_{}; \tilde{X}_\alpha)
                := \left\{ u \in L^2(I_{}; \tilde{X}_\alpha) : \bigintsss_{I_{}} \sum_{j \in \Z^2}  \frac{\left|\alpha_j \cdot \widehat{(  {\mathbf E}_{\alpha, T})}_j \right|^2}{|k^2 - | \alpha_j|^2|^{\nicefrac{1}{2}}} \dd \alpha  < \infty \right\}
                ,
            \]
            where the norm can be written as
            \[
                ||u||_{L^2_{\mathrm{w}}(I_{}; \tilde{X}_\alpha)}^2
                := \bigintsss_{I_{}} ||u(\alpha, \cdot)||^2_{\tilde{X}_\alpha} + \sum_{j \in \Z^2}  \frac{\left|\alpha_j \cdot \widehat{(  {\mathbf E}_{\alpha, T})}_j \right|^2 }{|k^2 - | \alpha_j|^2|^{\nicefrac{1}{2}}} \dd \alpha
                .
            \]
            Therefore, we can state the alternative problem as:
            
            \emph{For $\tilde{f} \in L^2(I_{} \times \Omega^R_0)$ we seek $\tilde{{\mathbf E}} \in L^2_{\mathrm{w}}(I_{}; \tilde{X}_\alpha)$, such that}
            \begin{subequations}\label{eq_transVarProblem}
                \begin{align}
                    &\int_{I_{}} \int_{\Omega_0^R} 
                    \mu_r^{-1} \left(\nabla \times \tilde{{\mathbf E}}_\alpha \right) \cdot \left( \nabla \times \overline{v}_\alpha \right) 
                    - k^2\epsilon_r \tilde{{\mathbf E}}_\alpha \cdot \overline{v}_\alpha
                    \dd x \dd \alpha
                    \\ 
                    & \quad \quad
                    + \int_{I_{}}  \int_{\Gamma_0^{R}} 
                    N_\alpha(  \tilde{{\mathbf E}}_{\alpha,T})  \cdot \overline{v}_{\alpha,T} 
                    - T_\alpha( \tilde{{\mathbf E}}_{\alpha,T})  \cdot   \overline{v}_{\alpha,T}
                    \dd S \dd \alpha
                    \\
                    &
                    = \int_{I_{}} \int_{\Omega_0^R} \tilde{f}_\alpha \cdot \overline{v_\alpha} \dd x\dd \alpha
                \end{align}
            \end{subequations}
            \emph{holds for all  $v \in L^2_{\mathrm{w}}(I_{}; \tilde{X}_\alpha)$.}
            
            The sesquilinear form is well-defined by construction and the following lemma shows the equivalence of the problems.
            \begin{lemma}
                The function  ${\mathbf E} \in X$ solves the variational problem~\eqref{eq_redProb} in the strip domain $\Omega^R$ if, and only if, $\tilde{{\mathbf E}} := (\JJ_{\R^2} {\mathbf E}) \in L^2_{\mathrm{w}}(I_{}; \tilde{X}_\alpha)$ solves the alternative variational problem~\eqref{eq_transVarProblem} with the right-hand side $\tilde{f}:= (\JJ_{\R^2} f)$.
            \end{lemma}
            \begin{proof}
                For the operator $T$ and the term in the volume integral part the equivalence can be proven analogously to \cite[Theorem~2]{Konschin19a}.
                We choose a cut-off function $\chi_{\{ ||\xi| - k | \geq \eta\}}$ for some $\eta > 0$ and consider $\mathcal{F}^{-1}\left(\chi_{\{ ||\xi| - k | \geq \eta\}} \mathcal{F}(N({\mathbf E}_T))\right)$. Then, we can show the equivalence for the operator $N$ analogously to $T$ and let  $\eta$ go to zero.
            \end{proof}
            
            At first, we  show the uniqueness, which is a direct consequence of the unique continuation property shown in \cite{Okaji2002}. We start by citing the corresponding result.
            
            \begin{proposition}\label{proposition_uniqueContienuationMaxwell}
                Let $U \subseteq \R^3$ be some domain and the parameter $\mu_r$ and $\epsilon_r$ functions in $W^{1, \infty}(\R^3)$. If ${\mathbf E}$ solves the Maxwell's equations~\eqref{eq_Maxwell3_1} for the right-hand side $f = 0$ and vanishes on an open set, then ${\mathbf E}$ vanishes everywhere in $U$.
            \end{proposition}
            \begin{proof}
                The results is proven in \cite[Theorem~2.3]{Okaji2002}.
            \end{proof}

            Thus, we can show the uniqueness of a solution to the problem \ref{prob_Maxwell}.
            \begin{lemma}\label{lemma_eindInt}
                If the Assumption~\ref{assumption_absorptionMaxwell} holds, then there  exists at maximum one solution to the problem in the integral form~\eqref{eq_transVarProblem} for every right-hand side.
            \end{lemma}
            \begin{proof}
                Let $\tilde{{\mathbf E}}$ be the solution to the problem~\eqref{eq_transVarProblem} for $\tilde{f} = 0$, then for all $v \in L^2_{\mathrm{w}}(I_{}; \tilde{X}_\alpha)$ it holds
                \begin{align*}
                    0
                    &= \int_{I_{}} \int_{\Omega_0^R} 
                    \mu_r^{-1} \left( \nabla \times \tilde{{\mathbf E}}_\alpha \right) \cdot \left(\nabla \times \overline{v}_\alpha \right)
                    - k^2\epsilon_r \tilde{{\mathbf E}}_\alpha \cdot \overline{v}_\alpha
                    \dd  x \dd  \alpha
                    \nonumber
                    \\ 
                    & \quad \quad
                    + \int_{I_{}}  \int_{\Gamma_0^{R}} 
                    N_\alpha(\tilde{{\mathbf E}}_{\alpha,T} )  \cdot  \overline{v}_{\alpha,T} 
                    - T_\alpha(\tilde{{\mathbf E}}_{\alpha,T} )  \cdot  \overline{v}_{\alpha,T} 
                    \dd  S \dd  \alpha
                    .
                \end{align*}
                Therefore, we can conclude using the inequalities in \eqref{eq_Tinequalities} and \eqref{eq_Ninequalities} that
                \begin{align*}
                    0
                    &= \int_{I_{}} \int_{\Omega_0^R} - k^2 (\Im \epsilon_r) |\tilde{{\mathbf E}}_\alpha|^2 \dd  x \dd  \alpha
                    + \Im \int_{I_{}}  \int_{\Gamma_0^{R}}  
                    \left(N_\alpha- T_\alpha\right)(\tilde{{\mathbf E}}_{\alpha,T} )  \cdot \overline{\tilde{{\mathbf E}}}_{\alpha,T} 
                    \dd  S\dd  \alpha
                    \\
                    &\leq \int_{I_{}} \int_{\Omega_0^R} - k^2 (\Im \epsilon_r) |\tilde{{\mathbf E}}_\alpha|^2 \dd  x \dd  \alpha
                    \leq 0
                \end{align*}
                holds, wherefrom $\tilde{{\mathbf E}}_\alpha = 0$ on $\{(\Im \epsilon_r) > 0 \}$ follows for almost all $\alpha \in I_{}$.
                Hence, we derive $\tilde{{\mathbf E}}_\alpha = 0$ on $\Omega^R_0$ for almost all $\alpha \in I_{}$ applying Proposition~\ref{proposition_uniqueContienuationMaxwell}.
            \end{proof}
            
            The variational problem~\eqref{eq_transVarProblem} is formulated with an additional integral surrounding the variational formulation.
            In this case we do not have any compact embedding results for the solution space $L^2_{\mathrm{w}}(I_{}; \tilde{X}_\alpha)$, and Fredholm theory is not applicable. For this reason, we are considering the problem pointwise in $\alpha \in \overline{I} \setminus \mathcal{A}$, for which we can decompose the solution space by the Helmholtz decomposition and derive compact embedding of the reduced problem. This will be our next step.
            
            The quasi-periodic variational problem for $\alpha \in \overline{I_{}} \setminus \mathcal{A}$ is given by:
            
            \emph{We seek $\tilde{{\mathbf E}}_\alpha \in \tilde{X}_\alpha$, such that}
            \begin{subequations}\label{eq_probWirklichPunkt}
                \begin{align}
                a_\alpha(\tilde{{\mathbf E}}_\alpha, v_\alpha)
                &:=\int_{\Omega_0^R} 
                \mu_r^{-1} \left(\nabla \times \tilde{{\mathbf E}}_\alpha \right) \cdot \left( \nabla \times\overline{ v }_\alpha \right)
                - k^2\epsilon_r \tilde{{\mathbf E}}_\alpha \cdot \overline{v}_\alpha
                \dd  x 
                \\
                &\quad\quad\quad
                +  \int_{\Gamma_0^{R}} 
                N_\alpha(\tilde{{\mathbf E}}_{\alpha,T} )  \cdot  \overline{v}_{\alpha,T} 
                -T_\alpha(\tilde{{\mathbf E}}_{\alpha,T} )  \cdot  \overline{v}_{\alpha,T} 
                \dd  S 
                \\
                &= f_\alpha(v_\alpha)
                \end{align}
            \end{subequations}
            \emph{holds for all  $v_\alpha \in \tilde{X}_\alpha$.}
            
        \subsection{Helmholtz decomposition of the solution space}
            In the following we  apply two different Helmholtz decompositions to the variational problem \eqref{eq_probWirklichPunkt}. The first decomposition is for reducing the solution space to some more regular subspace, and with the help of the  second decomposition, we include a boundary condition, which will be crucial for the decomposition of the differential operator into a coercive part and a compact perturbation.
            For the Helmholtz decomposition we consider the two following problems, where the first one is solved in $W:=\tilde{H}^1_\alpha(\Omega_0^R)$ and for the second one we seek the solution in  $W_0 := \{ w \in W : w = 0 \text{ on } \Gamma^R_0 \}$.
            We define the sesquilinear form $b^{(\varepsilon_r)}$ for all $\phi$, $v \in \tilde{X}_\alpha$ as
            \begin{equation*}
                b^{(\varepsilon_r)}(\phi, v) 
                := \int_{\Omega_0^R} k^2 \epsilon_r \phi \cdot \overline{ v} \dd  x 
                -   \int_{\Gamma_0^R}  \left( N_\alpha - T_\alpha \right)(\phi_{ T})   \cdot   \overline{v}_{ T} 
                \dd  S
                ,
            \end{equation*}
            which is well-defined for all functions in $\tilde{X}_\alpha$.
            For the first decomposition we seek a function in $W_0$. In this case the boundary terms of the sesquilinear form $b^{(\varepsilon_r)}$ can be omitted and we derive the following problem:
            \begin{lemma}\label{lemma_loesHilfs1}
                For  $F \in W_0'$ the problem 
                \begin{align*}
                    \divv \left(k^2\epsilon_r  \nabla w\right) 
                    &= -F
                    &&\text{in }
                    \Omega_0^R
                    \\
                    w &= 0
                    &&\text{on }
                    \Gamma_0^{0} \cup \Gamma_0^{R}
                \end{align*}
                has a unique solution $w \in W_0$.
            \end{lemma}
            \begin{proof}
                The corresponding variational problem is to find a $w \in W_0$, such that
                \[
                    b^{(\varepsilon_r)} \left(\nabla w, \nabla v \right) 
                    = \int_{\Omega_0^R} F \cdot \overline{v} \dd  x  
                    \quad\text{holds for all }
                    v \in W_0
                    .
                \]
                Since the trace of $W_0$-functions vanishes on the boundary $\Gamma_0^R \cup \Gamma_0^0$, the coercivity follows by the Poincar\'e inequality together with the estimation
                \[
                    \Re b^{(\varepsilon_r)}(\nabla w, \nabla w) 
                    = \Re \int_{\Omega_0^R} k^2 \epsilon_r |\nabla { w}|^2 \dd  x 
                    \geq k^2 \varepsilon_0 \,||\nabla w||^2_{L^2(\Omega_0^R)}
                    .
                \]
                Hence, the unique existence of the solution follows by the lemma of Lax-Milgram.
            \end{proof}
            
            For the second decomposition, we seek  a function in $W$ which possesses a special boundary condition.
            \begin{lemma}\label{lemma_loesHilfs2}
                For $\alpha \in \overline{I}\setminus \mathcal{A}$ and $G \in H_\alpha^{\nicefrac{-1}{2}}(\Gamma_0^{R})$
                the boundary value problem
                \begin{align*}
                    \divv \left(k^2\epsilon_r  \nabla w\right) 
                    &= 0
                    &&\text{in }
                    \Omega_0^R
                    \\
                    \frac{\partial w}{\partial x_3} +k^{-2} \divv_T( N_\alpha - T_\alpha)\left(\nabla_T  w\right)
                    &=  k^{-2} G
                    &&\text{on }
                    \Gamma_0^{R}
                    \\
                    w
                    &=  0
                    &&\text{on }
                    \Gamma_0^{0}
                \end{align*}
                is uniquely solvable in $W$.
                If $G$ is an element of $H_\alpha^{\nicefrac{1}{2}}(\Gamma_0^{R})$, then the solution is  an element of $ H_\alpha^2(\Omega_0^R)$.
            \end{lemma}
            \begin{proof}
                Since the functions are $\alpha$-quasi-periodic on the boundary, it holds the identity $\hat{(\nabla_T w)}_j = -\ii\alpha_j \,\hat{w}_j$ for the  $j$th Fourier coefficient $\hat{w}_j$ of $w \in W$. Thus, it holds
                \begin{subequations}\label{eq_nablaTw}
                    \begin{align*}
                        k^{2}\Re \langle B_\alpha w, w \rangle :=
                        &-   \Re    \int_{\Gamma_0^R}  \left( N_\alpha - T_\alpha \right)(\nabla_T  w)  \cdot \nabla_T  \overline{w} 
                        \dd  S
                        \\
                        &= \sum_{|\alpha_j| > k} \frac{1}{\sqrt{|\alpha_j|^2 - k^2 }} \left( |\alpha_j|^4 |\hat{w}_j|^2  + (k^2 - |\alpha_j|^2 ) |\alpha_j|^2 |\hat{w}_j|^2  \right)
                        \\
                        &=k^2 \sum_{|\alpha_j| > k} \frac{|\alpha_j|^2}{|k^2 - |\alpha_j|^2|^{\nicefrac{1}{2}}}\, |\hat{w}_j|^2 \geq 0
                        .
                    \end{align*}
                \end{subequations}
                Therefore, the existence of the unique solution follows analogously to Lemma~\ref{lemma_loesHilfs1}.
                
                The second part of the statement regards the regularity of the solution can be proven analogously to Section~2.3 in \cite{Grisv1985}, which argumentation we  sketch here.
                Following the argumentation in \cite{Grisv1985}, it is sufficient to show 
                the estimation
                \[
                    ||v||_{H_\alpha^2(\Omega)} \leq C
                    \left( 
                    ||v||_{H_\alpha^1(\Omega)} + ||G ||_{H_\alpha^{\nicefrac{1}{2}}(\Omega)}
                    \right)
                \]
                for the solution $v \in H_\alpha^2(\Omega)$, $\Omega := (-\pi, \pi)^2 \times (-\infty, R)$, of the problem
                \begin{subequations}\label{eq_RegW}
                    \begin{align}
                        -\Delta v  + v 
                        &= 0
                        &&\text{in }
                        \Omega
                        \label{eq_RegWa}
                        \\
                        \frac{\partial v  }{\partial x_3} + B_\alpha v  
                        &= G
                        &&\text{on }
                        \Gamma_0^R
                        .
                    \end{align}
                \end{subequations}
                To construct the solution to this problem, one can solve an ordinary differential equation for the Fourier coefficients and derive the solution
                \[
                    v 
                    = \sum_{j \in \Z^{2} }\frac{\hat{G}_j \sqrt{k^2-|\alpha_j|^2} }{\sqrt{k^2-|\alpha_j|^2}\sqrt{1+|\alpha_j|^2} + {\ii  |\alpha_j|^2}} \, e^{-\ii \alpha_j \cdot \tx  +\sqrt{1+|\alpha_j|^2} (x_3-R)}
                    \in H_\alpha^2(\Omega)
                    ,
                \]
                where the $H_\alpha^2(\Omega)$ norm can be estimated by the $H_\alpha^{\nicefrac{1}{2}}(\Gamma^R_0)$ norm of $G$ analogously to \cite[Lemma~2.2]{Chand2005}.
                Using a convolution with the solution operator for the problem 
                $-\Delta v  + v = f$ in $(-\pi, \pi)^2 \times  \R$, we derive the estimation of the $H^2(\Omega)$ norm of the solution to the problem with an additional right-hand side $f \in L^2(\Omega)$ in \eqref{eq_RegW} (compare \cite[Lemma~2.3.2.4]{Grisv1985}). 
                
                Now, we split the domain $\Omega_0^R$ into  $\Omega_0^{R-\delta} \cup D$, where $D:= \Omega_0^{R} \setminus \overline{\Omega_0^{R-2\delta} }$ for a sufficient small $\delta > 0$.
                By \cite[Theorem~4.18]{McLea2000}), we conclude that $w$ is $H^2$-regular on  $\Omega_0^{R-\delta}$. Choosing a cut-off function $\chi \in C^{\infty}(-\infty, R)$ with $\chi(x_3) = 0$ for $|x_3-R| \geq 2\delta$ and $\chi(x_3) = 1$ for $|x_3-R| \leq \delta$, we can extend the function $\chi w$  with zero to $\Omega$. The function $\chi w$ solves $-\Delta (\chi w) +\chi w  = f \in L^2(D)$ with $G$ as the boundary condition. Thus, we can conclude the $H^2$-regularity of the solution $w$ in  $\Omega_0^R \setminus \overline{\Omega_0^{R-\delta}}$, and in consequence, in the whole domain $\Omega_0^R$.
            \end{proof}
            
            Since the permittivity $\varepsilon_r$ is constant near the boundary with $\varepsilon_r=1$, we can define the trace $u_3$ on $\Gamma_0^{R}$ for every function $u \in \tilde{X}_\alpha$, which fulfills $\divv (\varepsilon_r u) = 0$ in $\Omega_0^R$, and it holds $u_3 \in H^{\nicefrac{-1}{2}}(\Gamma_0^{R})$. Therefore, the four spaces
            \begin{equation*}
                \tilde{Y}_\alpha
                := \left\{ u \in \tilde{X}_\alpha : \divv(\epsilon_r u) = 0
                \text{ in }
                \Omega^R_0
                \right\}
                \quad\text{and}\quad
                \tilde{Y}_\alpha^{\perp} 
                := \left\{ u \in \tilde{X}_\alpha : u = \nabla w,\ w \in W_0
                \right\}
            \end{equation*}
            as well as
            \begin{equation*}
                {Y}_\alpha
                := \left\{ u \in \tilde{Y}_\alpha : k^2{u}_{\alpha,3} = -\divv_T \left[ ( N_\alpha - T_\alpha) u_{\alpha, T} \right]
                \in H^{\nicefrac{-1}{2}}(\Gamma_0^{R}) 
                \right\}
            \end{equation*}
            and
            \begin{equation*}
                {Y}_\alpha^{\perp} 
                := \left\{ u \in \tilde{Y}_\alpha : u = \nabla w,\ w \in W
                \right\}
            \end{equation*}
            are well-defined.
            \begin{lemma}\label{lemma_ZerlegungAlpha}
                The subspaces $\tilde{Y}_\alpha$ and $\tilde{Y}_\alpha^{\perp}$ of $\tilde{X}_\alpha$ are closed and  $\tilde{X}_\alpha$ can be decomposed into
                $
                \tilde{X}_\alpha
                = \tilde{Y}_\alpha \oplus \tilde{Y}_\alpha^{\perp}
                .
                $
            \end{lemma}
            \begin{proof}
                (i) Closeness of $\tilde{Y}_\alpha^{\perp}$: 
                Let $\{ \nabla \tilde{w}_n \}_{n \in \N}$ be a Cauchy sequence in $\tilde{Y}_\alpha^{\perp}$. For $\nabla v \in \tilde{Y}_\alpha^{\perp} $ the norm of $\tilde{X}_\alpha$ reduces to 
                $
                    ||\nabla v ||_{L^2(\Omega_0^R)} 
                    .
                $
                Hence, the sequence $\{  \tilde{w}_n \}_{n \in \N}$ is a Cauchy sequence in the closed space $W_0$ and possesses the limit $\tilde{w} \in W_0$.
                The norm equivalence implies that the sequence $\{ \nabla \tilde{w}_n \}_{n \in \N}$ convergence in $\tilde{X}_\alpha$ against $ \nabla \tilde{w}$ for $n \to \infty$, which is an element of $\tilde{Y}_\alpha^{\perp}$.
                
                (ii) Closeness of $\tilde{Y}_\alpha$: 
                We  show that $\tilde{Y}_\alpha$ is the null space of the operator $\tilde{P} \in \mathcal{L}(\tilde{X}_\alpha, W_0')$, $\tilde{P}\colon \tilde{u} \mapsto b^{(\varepsilon_r)}(\tilde{u}, \nabla  \cdot)$, which implies the  closeness.
                Obviously the   $\tilde{Y}_\alpha$ is a subspace of the null space $\mathcal{N} (\tilde{P})$. Hence, let $\tilde{u} \in \mathcal{N} (\tilde{P})$ be some function of the null space, then for all $\phi \in C^\infty_0 (\Omega_0^R)$ it holds $\nabla \phi \in X_\alpha$ and 
                \[
                    0
                    = b^{(\varepsilon_r)}(\tilde{u}, \nabla \phi)
                    = \int_{\Omega_0^R}  k^2 \varepsilon_r \tilde{u}\cdot \nabla \overline{\phi} \dd x
                    .
                \]
                Therefore, it holds $\divv (\varepsilon_r \tilde{u}) = 0$ in the distributional sense and we conclude $\tilde{u} \in \tilde{Y}_\alpha^\perp$.
                
                (iii) Decomposition of $\tilde{X}_\alpha$: 
                We choose ${\mathbf E} \in \tilde{X}_\alpha$, then there exists the unique solution $\tilde{w} \in W_0$ of the variational problem 
                \[
                    b^{(\varepsilon_r)}(\nabla \tilde{w},\nabla v)
                    = b^{(\varepsilon_r)}({\mathbf E},\nabla  v)
                    \quad\text{for all }
                    v \in W_0
                \]
                by Lemma~\ref{lemma_loesHilfs1}.
                For the function $\tilde{u} := {\mathbf E} - \nabla \tilde{w} \in \tilde{X}_\alpha$, it holds $\tilde{u}_T \in TH^{\nicefrac{1}{2}}(\Gamma_0^R)$ and
                \begin{equation*}
                    0
                    = b^{(\varepsilon_r)}(\tilde{u}, \nabla  v)
                    = \int_{\Omega_0^R} k^2 \epsilon_r \tilde{u} \cdot \nabla \overline{v}\dd x
                \end{equation*}
                for all $v \in W_0$. 
                Thus, the function $\tilde{u}$ is an element of $\tilde{Y}_\alpha$.
                
                (iv) Uniqueness of the decomposition: 
                Let $\tilde{u} = \nabla \tilde{w}$ be in $ \tilde{Y}_\alpha \cap \tilde{Y}_\alpha^{\perp}$. We choose $F = \divv (\epsilon_r \tilde{u}) = 0$, such that Lemma~\ref{lemma_loesHilfs1} implicates $\tilde{w} = 0$. Consequently, we conclude  $\tilde{u} = 0$.
            \end{proof}
            
            For $\varepsilon_r \in W^{1, \infty} (\Omega_0^R)$ and $\tilde{u} \in \tilde{Y}_\alpha$ it holds the estimation
            \begin{equation}\label{eq_divLeqU}
                || \divv \tilde{u}||_{L^2(\Omega_0^R)}
                \leq \left|\left| \frac{\nabla \varepsilon}{\varepsilon}\right|\right|_{L^\infty(\Omega_0^R)} || \tilde{u}||_{L^2(\Omega_0^R)} 
                \leq \frac{||\varepsilon_r ||_{W^{1,\infty}(\Omega_0^R)}}{\varepsilon_0} \, || \tilde{u}||_{L^2(\Omega_0^R)} 
                ,
            \end{equation}
            and hence, it follows by Theorem~\ref{satz_regGebiet} that $\tilde{Y}_\alpha$ is a subspace of $\tilde{H}_\alpha^1(\Omega^R)^3$ and the norms of $\tilde{X}_\alpha$ and of $H^1(\Omega_0^R)^3$ are equivalent on  $\tilde{Y}_\alpha$. 
            In the next step we apply a second Helmholtz decomposition to $\tilde{Y}_\alpha$ to get an additional boundary condition into the solution space.
            
            \begin{lemma}\label{lemma_ZerlegungAlpha2}
                Choose $\varepsilon_r \in W^{1, \infty} (\Omega_0^R)$, then the subspaces ${Y}_\alpha$ and ${Y}_\alpha^{\perp}$ of $\tilde{Y}_\alpha$ are closed and the space $\tilde{Y}_\alpha$ can be decomposed into
                $
                \tilde{Y}_\alpha
                = {Y}_\alpha \oplus {Y}_\alpha^{\perp}
                .
                $
            \end{lemma}
            \begin{proof}
                (i) The argumentation for the closeness of $Y_\alpha^\perp$ is analogously to the proof of Lemma~\ref{lemma_ZerlegungAlpha}, with the only difference that the $\tilde{X}_\alpha$ norm for  $\nabla v \in {Y}_\alpha^{\perp}$ reduces to
                \[
                    ||\nabla v ||^2_{\tilde{X}_\alpha}
                    =||\nabla v ||^2_{L^2(\Omega_0^R)} + || \nabla_T v ||^2_{TH^{\nicefrac{1}{2}}(\Gamma_0^R)} 
                \]
                and the Cauchy sequence converges in $\{ u \in W : \nabla_T u \in TH_\alpha^{\nicefrac{1}{2}}(\Gamma_0^{R}) \}$.
                
                (ii) We define the operator $P \in \mathcal{L}(\tilde{Y}_\alpha, W')$ as $P\colon u \mapsto b^{(\varepsilon_r)}(u, \nabla  \cdot)$ and show 
                $Y_\alpha = \mathcal{N}(P)$.
                For $u \in \mathcal{N}(P)$ it holds $\divv (\varepsilon_r u ) = 0$ and $u_3 \in H^{\nicefrac{1}{2}}(\Gamma^R_0)$ is well-defined.
                Choosing some $\psi \in C^\infty_0 (\R^3)$ which does not vanish on the boundary,  we can apply the Gaussian theorem and derive
                \[
                    0
                    = b^{(\varepsilon_r)}(u, \nabla \psi)
                    = \int_{\Gamma_0^R} k^2 u_{ 3}\, \overline{\psi} - ( N_\alpha - T_\alpha) u_{ T} \cdot \nabla_T \overline{\psi} \dd S
                    .
                \]
                Because of the arbitrary choice of $\psi \in C^\infty_0 (\R^3)$ the function $u$ has to be an element of  ${Y}_\alpha$.
                
                (iii) Decomposition of $\tilde{Y}_\alpha$: 
                We choose $\tilde{u} \in \tilde{Y}_\alpha$ and the right-hand side $G := k^2\tilde{u}_3 + \divv_T \left[ ( N_\alpha - T_\alpha) \tilde{u}_{T} \right]$. Then $G \in H_{\alpha}^{\nicefrac{-1}{2}}(\Gamma_0^R)$, since for all $\phi \in H_{\alpha}^{\nicefrac{1}{2}}(\Gamma_0^R)$ it holds
                \begin{align*}
                    & \left| \int_{\Gamma_0^R}  \left( N_\alpha - T_\alpha \right)(\tilde{u}_T  w)  \cdot \nabla_T  \overline{\phi} \right|
                    \dd  S
                    \\
                    &\leq  C \sum_{j \in \Z^2} \left(1+|j|^2\right)^{\nicefrac{-1}{2}} \left|  (\alpha_j \cdot \tilde{u}_T) \,|\alpha_j|^2\, \overline{\phi}  + (k^2 - |\alpha_j|^2 ) (\alpha_j \cdot \tilde{u}_T)\, \overline{\phi} \right|
                    \\
                    &\leq C ||\divv_T \tilde{u}_T||_{H_\alpha^{\nicefrac{-3}{2}}(\Gamma_0^R)} ||\phi||_{H_\alpha^{\nicefrac{1}{2}}(\Gamma_0^R)}
                    \\
                    &\leq C ||\tilde{u}_T||_{TH_\alpha^{\nicefrac{-1}{2}}(\Gamma_0^R)} ||\phi||_{H_\alpha^{\nicefrac{1}{2}}(\Gamma_0^R)}
                    .
                \end{align*}
                In Lemma~\ref{lemma_loesHilfs2} we showed that there exists a unique solution $w \in W$ to the variational problem
                \[
                    b^{(\varepsilon_r)}(\nabla w,\nabla v)
                    = b^{(\varepsilon_r)}(\tilde{u},\nabla  v)
                    \quad\text{for all }
                    v \in W
                    .
                \]
                Using the assumption $\varepsilon_r \in W^{1, \infty}(\Omega^R)$ and the estimation~\eqref{eq_divLeqU}, we have the inclusion $\tilde{Y}_\alpha \subseteq \tilde{H}^1_\alpha(\Omega_0^R)^3$.
                Therefore, the right-hand side $G$ is actually an element of $ H_\alpha^{\nicefrac{1}{2}}(\Gamma^R_0)$ and by Lemma~\ref{lemma_loesHilfs2}, we derive $w \in H^2(\Omega_0^R)$.
                In particular, the function $u := \tilde{u} - \nabla w \in \tilde{Y}_\alpha$ satisfies
                \begin{equation*}
                    0
                    = b^{(\varepsilon_r)}(u, \nabla  v)
                    = \int_{\Omega_0^R} k^2 \epsilon_r u \cdot \nabla \overline{v}\dd x
                    -   \int_{\Gamma_0^R}  \left( N_\alpha - T_\alpha \right)(\tilde{u}_{T}) \cdot \nabla_T \overline{v} 
                    \dd  S
                \end{equation*}
                for all $v \in W$, and in consequence, $u \in Y_\alpha$.
                
                (iv) The uniqueness of the decomposition follows by Lemma~\ref{lemma_loesHilfs2}, if we choose $G= k^{2} u_{3} + \divv_T( N_\alpha - T_\alpha)(u_{T}) = 0$ for $u = \nabla w \in Y_\alpha \cap Y_\alpha^{\perp}$.
            \end{proof}
            
            Applying both Helmholtz decompositions to our variational problem, we can split a function $\tilde{X}_\alpha$ into three unique functions of $Y_\alpha$,  $Y_\alpha^\perp$ and $\tilde{Y}_\alpha^{\perp}$.
            Let ${\mathbf E} = u_{\mathbf E} + \nabla w_{\mathbf E} + \nabla \tilde{w}_{\mathbf E} \in \tilde{X}_\alpha$ be the solution to problem \eqref{eq_probWirklichPunkt} and let $v = u_v + \nabla w_v+ \nabla \tilde{w}_v\in \tilde{X}_\alpha$ be a test function, where $u_{\mathbf E}, u_v \in Y_\alpha$, $\nabla w_u, \nabla w_v \in Y_\alpha^{\perp}$ and $\nabla \tilde{w}_u, \nabla \tilde{w}_v \in \tilde{Y}_\alpha^{\perp}$.
            We conclude by Lemma~\ref{lemma_loesHilfs1} and by Lemma~\ref{lemma_loesHilfs2} that $w_{\mathbf E} \in W$ and $\tilde{w}_{\mathbf E} \in W_0$ are the unique solutions to
            \begin{equation*}
                b^{(\varepsilon_r)}(\nabla w_{\mathbf E}, \nabla  w_v) 
                = \int_{\Omega_0^R} {f}_\alpha \cdot \nabla w_v \dd  x
                ,\quad
                b^{(\varepsilon_r)}(\nabla \tilde{w}_{\mathbf E}, \nabla  \tilde{w}_v) 
                = \int_{\Omega_0^R} {f}_\alpha \cdot \nabla \tilde{w}_v \dd  x
            \end{equation*}
            for all $w_v \in W$ and $\tilde{w}_v \in W_0$.
            Therefore, the problem can be reduced to:

            \emph{We seek $u_{\mathbf E} \in Y_\alpha$, such that }
            \begin{equation}\label{eq_redProb}
                a_\alpha(u_{\mathbf E}, u_v)
                = g_\alpha(u_v)
                \quad\text{\emph{holds for all }}
                u_v \in Y_\alpha
                ,
            \end{equation}
            where the right-hand side is given by 
            \begin{align*}
                g_\alpha(u_v)
                &
                := \int_{\Omega_0^R} {f}_\alpha \cdot \overline{u}_v \dd  x - a_\alpha(\nabla w_{\mathbf E}, u_v) - a_\alpha(\nabla \tilde{w}_{\mathbf E}, u_v)
                \\
                &
                = \int_{\Omega_0^R} \left({f}_\alpha+ k^2  \epsilon_r \nabla (w_{\mathbf E}+\tilde{w}_{\mathbf E}) \right) \cdot \overline{u}_v \dd  x 
                -  \int_{\Gamma_0^{R}}  
                          \left( N_\alpha - T_\alpha \right)(\nabla_T w_{\mathbf E})  \cdot \overline{u}_v
                    \dd  S
                .
            \end{align*}

        \subsection{Unique existence of the solution to the quasi-periodic problem}
            In this section we consider the reduced quasi-periodic variational problem and show the unique existence of the solution.
            Let $\alpha \in \overline{I_{}} \setminus \mathcal{A}$ be fixed and choose a sufficient large $\rho > 0$. We define the sesquilinear form  $a^\rho_\alpha$ for all $u_\alpha$, $v_\alpha \in {Y}_\alpha$ as
            \begin{align*}
                a^\rho_\alpha(u_\alpha,v_\alpha)
                &:= \int_{\Omega_0^R} 
                \mu_r^{-1} \left(\nabla \times u_\alpha \right) \cdot \left( \nabla \times \overline{v }_\alpha \right)
                +  \rho\, u_\alpha \cdot \overline{v}_\alpha
                \dd  x 
                \\
                & \quad \quad \quad \quad 
                +
                \int_{\Gamma_0^{R}} 
                N_\alpha(u_{\alpha,T} )  \cdot  \overline{v}_{\alpha,T} 
                -T_\alpha(u_{\alpha,T} )  \cdot  \overline{v}_{\alpha,T} 
                \dd  S
                \\ 
                & \quad \quad \quad \quad 
                +
                C(k^2, \alpha) 
                \sum_{j \in \Z^2} (1+|\alpha_j|^2)^{\nicefrac{-1}{2}}\, \widehat{(u_{\alpha,T})}_j \cdot  \overline{\widehat{(v_{\alpha,T})}_j}
                ,
            \end{align*}
            where the constant $C(k^2, \alpha) $ is given by
            \[
                C(k^2, \alpha) 
                := \frac{k^2}{2 \pi}\, \sup_{j \in \Z^2 } \frac{(1 + |\alpha_j|^2)^{\nicefrac{1}{2}}}{|k^2 - |\alpha_j|^2|^{\nicefrac{1}{2}}}
                .
            \]
            Thus, we can write the sesquilinear form $a_\alpha $ as
            \begin{align*}
                a_\alpha(u_\alpha,v_\alpha)
                &= 
                a^\rho_\alpha(u_\alpha,v_\alpha) 
                - \int_{\Gamma_0^{R}} (\rho+ k^2\epsilon_r)\, u_\alpha \cdot \overline{v}_\alpha \dd x 
                \\
                &   \quad\quad\quad
                -
                C(k^2, \alpha) 
                \sum_{j \in \Z^2} (1+|\alpha_j|^2)^{\nicefrac{-1}{2}}\, \widehat{(u_{\alpha,T})}_j \cdot  \overline{\widehat{(v_{\alpha,T})}_j}
                .
            \end{align*}

            \begin{theorem}\label{satz_koerziv}
                For $\varepsilon_r \in W^{1, \infty} (\Omega_0^R)$ and a sufficient large  $\rho > 0$ the sesquilinear form $a^\rho_\alpha$ is coercive on  $\tilde{H}_\alpha^1(\Omega_0^R)^3$, and, in particular, on  ${Y}_\alpha$.
            \end{theorem}
            \begin{proof}
                For the boundary term it holds
                \begin{align*}
                    &\frac{1}{2 \pi}\sum_{j \in \Z^2,\,|\alpha_j| > k} -\left(| \alpha_j|^2 - k^2 \right)^{\nicefrac{-1}{2}}\, \left|\alpha_j \cdot \widehat{(u_T)}_{j}\right|^2  + \left(|\alpha_j|^2 - k^2\right)^{\nicefrac{1}{2}} \,\left|\widehat{(u_T)}_{j}\right|^2
                    \\
                    &\geq 
                    \frac{1}{2 \pi}\sum_{j \in \Z^2,\,|\alpha_j| > k} \left(| \alpha_j|^2 - k^2\right)^{\nicefrac{-1}{2}} \left(-|\alpha_j|^2+ | \alpha_j|^2 - k^2 \right)\, \left|\widehat{(u_T)}_{j}\right|^2
                    \\
                    &=
                    \frac{- k^2}{2 \pi}\sum_{j \in \Z^2,\,|\alpha_j| > k}  \frac{\left|\widehat{(u_T)}_{j}\right|^2}{\left(| \alpha_j|^2 - k^2 \right)^{\nicefrac{1}{2}}}
                    ,
                \end{align*}
                and, in consequence, we have the estimation
                \begin{align*}
                    \Re  \int_{\Gamma_0^{R}}
                    \left[ N_\alpha -T_\alpha \right](u_{\alpha,T} )  \cdot  \overline{u}_{\alpha,T}
                    \dd  S
                    \geq
                    -C(k^2, \alpha)\, ||u_{\alpha,T}||^2_{H_\alpha^{\nicefrac{-1}{2}}(\Gamma_0^{R})}
                    .
                \end{align*}
                The assumption $\varepsilon_r \in W^{1, \infty} (\Omega_0^R)$ implicates
                $
                    ||\divv u_\alpha||^2_{L^2(\Omega_0^R)} - C_2||u_\alpha||^2_{L^2(\Omega_0^R)} 
                    \leq 0
                    ,
                $
                wherefrom we derive
                \begin{align*}
                    &\Re \int_{\Omega_0^R} 
                    \mu_r^{-1} \left( \nabla \times u_\alpha \right) \cdot \left( \nabla \times \overline{u }_\alpha \right)
                    +   \rho\, u_\alpha \cdot \overline{u}_\alpha 
                    \dd  x 
                    \\
                    &
                    \geq ||\mu_r||_{L^\infty(\Omega_0^R)}^{-1} ||\nabla \times u_\alpha||^2_{L^2(\Omega_0^R)^3} +   \rho\,||u_\alpha||^2_{L^2(\Omega_0^R)^3}
                    \\
                    &\geq C_1 ||\nabla \times u_\alpha||^2_{L^2(\Omega_0^R)^3} + C_1||\divv \cdot   u_\alpha  ||^2_{L^2(\Omega_0^R)} + (\rho-C_1C_2) ||u_\alpha||^2_{L^2(\Omega_0^R)^3}
                    .
                \end{align*}
                An analogous computation to Theorem~\ref{satz_regGebiet} gives the identity
                \begin{equation*}
                    \int_{\Omega_0^R} | \nabla \times v |^2 + |\divv v|^2 \dd  x
                    = \sum_{j=1}^3 \int_{\Omega_0^R} |\nabla v_j|^2 + 2\, \Re    \int_{\Gamma_0^R}   (\divv_T v_T )\,\overline{v}_3 \dd  S 
                \end{equation*}
                for $v \in \tilde{H}_{\alpha}^1(\Omega_0^R)^3$, wherefrom it follows
                \begin{align*}
                    &\Re \int_{\Omega_0^R} 
                    \mu_r^{-1} \left( \nabla \times u_\alpha \right) \cdot \left( \nabla \times \overline{u }_\alpha \right)
                    + \rho\, u_\alpha \cdot \overline{u}_\alpha 
                    \dd  x 
                    \\
                    &
                    \geq C_1\, ||u_\alpha||^2_{H^1(\Omega_0^R)^3}  + (\rho- C_2   C_1)\, ||u_\alpha||^2_{L^2(\Omega_0^R)^3}
                    \\
                    & \quad\quad\quad
                    + 2  C_1\, \Re   \int_{\Gamma_0^R}  (\divv_T u_{\alpha,T}) \overline{u}_{\alpha,3} \dd  S
                    .
                \end{align*}
                Let $( u )^{\hat{}}_j$ denote the $j$th Fourier coefficient of some function $u \in L^2(\Gamma^R_0)$. Considering the boundary condition of the space $Y_\alpha$, we compute
                \begin{align*}
                    k^2   \hat{(u_{\alpha, 3})}_{j} 
                    &=   - \left(\divv_T\left( N_\alpha - T_\alpha\right)( u_{\alpha, T})\right)^{\hat{}}_j
                    \\
                    &=  \ii    \alpha_j \cdot \left( \frac{ -\ii \alpha_j}{\sqrt{k^2-|\alpha_j|^2}} \left( \alpha_j \cdot \widehat{(u_{\alpha, T})}_j\right) - \ii \sqrt{k^2 - |\alpha_j|^2}\, \widehat{(u_{\alpha, T})}_j \right)
                    \\
                    &= \frac{ 1  }{\sqrt{k^2-|\alpha_j|^2}}   \left( |\alpha_j|^2  \left( \alpha_j \cdot \widehat{(u_{\alpha, T})}_j\right)  +\left(k^2 - |\alpha_j|^2\right) \left(\alpha_j \cdot \widehat{(u_{\alpha, T})}_j \right) \right)
                    \\
                    &= k^2 \frac{ -\ii \alpha_j \cdot \widehat{(u_{\alpha, T})}_j }{ -\ii  \sqrt{k^2-|\alpha_j|^2}}
                    .
                \end{align*}
                Therefore, the boundary term is non-negative, since
                \begin{equation*}
                    2\, C_1\, \Re  \int_{\Gamma_0^R}  (\divv_T u_{\alpha, T})\, \overline{u}_{\alpha, 3} \dd  S
                    = 2\, C_1 \sum_{j \in \Z^2,\,|\alpha_j| > k}   \frac{1}{\sqrt{|\alpha_j|^2 - k^2}} \,|\alpha_j \cdot u_{\alpha, T}  |^2  
                    \geq 0
                    .
                \end{equation*}
                If we put everything together, we derive the estimation 
                \begin{align*}
                    \Re a^\rho_\alpha(u_\alpha, u_\alpha)
                    \geq C\, ||u_\alpha||^2_{H^1(\Omega_0^R)^3} + C(\rho)  \,|| u_\alpha||^2_{L^2(\Omega_0^R)^3}
                    \geq c \,||u_{\alpha}||^2_{\tilde{X}_\alpha}
                    ,
                \end{align*}
                and hence, the sesquilinear form is coercive on $\tilde{Y}_\alpha$.
            \end{proof}
            
            Thus, we are prepared to show the unique existence of the solution to the reduced quasi-periodic problem.
            
            \begin{lemma}\label{lemma_perRWPLoesung}
                If the Assumption~\ref{assumption_absorptionMaxwell} holds, then for all  $\alpha \in \overline{I_{}} \setminus \mathcal{A}$ the problem
                \begin{equation*}
                    a_\alpha(u_\alpha, v_\alpha)
                    = g_\alpha(v_\alpha)
                    \quad\text{for all  }
                    v_\alpha \in {Y}_\alpha
                \end{equation*}
                is uniquely solvable.
            \end{lemma}
            \begin{proof}
                Because of Theorem~\ref{satz_koerziv} and the compact embedding of $\tilde{H}^1_\alpha(\Omega_0^R)$ in $L^2(\Omega_0^R)$ (see \cite[Theorem~3.27]{McLea2000}), we can split the sesquilinear form into a coercive part and a compact perturbation.
                Thus, it remains to show the uniqueness and apply the Fredholm alternative for the existence of the solution.
                The uniqueness can be shown analogously to Lemma~\ref{lemma_eindInt} by applying Proposition~\ref{proposition_uniqueContienuationMaxwell}, since
                \[
                    0
                    = \int_{\Omega_0^R} - k^2 (\Im \epsilon_r) \, |u_\alpha|^2 \dd  x 
                    + \Im    \int_{\Gamma_0^{R}} 
                    \left( N_\alpha - T_\alpha \right) (u_{\alpha,T} )\cdot \overline{u}_{\alpha,T} 
                    \dd  S
                    \leq 0
                \]
                holds.
            \end{proof}
            
        \subsection{Constructing the solution}
            Having the existence theory for the quasi-periodic problems, we want to construct the solution to the problem in the integral form \eqref{eq_transVarProblem}. The following theorem summarizes this subsection.
            \begin{theorem}\label{satz_pktweise1}
                The existence of unique solutions $u_\alpha \in \tilde{X}_\alpha$ for all $\alpha \in \overline{I_{}} \setminus \mathcal{A}$ to the $\alpha$-quasi-periodic problem~\eqref{eq_probWirklichPunkt} implicates the existence of the unique solution $\tilde{u} \in L^2_{\mathrm{w}}(I_{}; \tilde{X}_\alpha)$, $\tilde{u}(\alpha, \cdot) := u_\alpha$, to the problem~\eqref{eq_transVarProblem} in the integral form.
                Furthermore, it holds 
                \begin{equation}\label{eq_abschaetzung}
                    ||\tilde{u}||_{L^2_{\mathrm{w}}(I_{};\tilde{X}_\alpha)}
                    \leq c\, ||\tilde{f}||_{L^2(I_{} \times \Omega_0^R)}
                    .
                \end{equation}
            \end{theorem}
            
            For the proof we  show that the quasi-periodic solution operator $L_\alpha$, which is only defined for $\alpha \in \overline{I}_{}\setminus \mathcal{A}$, can be extended continuously to $\overline{I}_{}$. Therefore, we can find a global constant $C$, such that $\sup_{\alpha \in \overline{I}} ||L_\alpha||_{L^2(I;X_\alpha)} \leq C$ holds.
            Afterwards, we still have to show that the estimation~\eqref{eq_abschaetzung} holds for the weighted space $ L^2_{\mathrm{w}}(I_{}; \tilde{X}_\alpha)$. For the extension we  utilize the Sherman-Morrison-Woodbury formula, which we  prove first.
            \begin{theorem}{}\label{satz_SMWFormel}
                Let $H_1$ and $H_2$ two Hilbert spaces and $S \in \mathcal{L}(H_1)$ as well as $D \in \mathcal{L}(H_2)$ two invertible bounded operators.
                Further, choose two linear and continuous operators $Z_1 \in \mathcal{L}(H_1, H_2)$ and $Z_2 \in \mathcal{L}(H_1, H_2)$, such that
                \[
                    B
                    := S + Z_2^* D Z_1 
                    \in \mathcal{L}(H_1)
                \quad\text{and}\quad
                    G
                    :=
                    D^{-1} + Z_1 S^{-1} Z_2^*
                    \in \mathcal{L}(H_2)
                \]
                are invertible. Then the inverse of $B$ can be represented by
                \begin{equation}\label{eq_SMW}
                    B^{-1}
                    = 
                    S^{-1} - S^{-1} Z_2^* \left(D^{-1} + Z_1 S^{-1} Z_2^* \right)^{-1} Z_1 S^{-1}
                    .
                \end{equation}
            \end{theorem}
            \begin{proof}
                We call the operator on the right-hand side of \eqref{eq_SMW} as $C$. We assumed that $S$, $D$, $B$ and $G$ are continuously invertible operators, and hence, $C$ is also continuously invertible.
                We call $I_1 \in \mathcal{L}(H_1)$ and $I_2 \in \mathcal{L}(H_1)$ the two identity operators on $H_1$, or  on $H_2$, respectively. Thus, we compute
                \begin{align*}
                    C B
                    &= \left(S^{-1} - S^{-1} Z_2^* \left(D^{-1} + Z_1 S^{-1} Z_2^*\right)^{-1} Z_1 S^{-1} \right) (S + Z_2^* D Z_1 )
                    \\
                    &= I_1 + S^{-1}Z_2^* D Z_1 - S^{-1} Z_2^* G^{-1} Z_1 
                    - S^{-1} Z_2^* G^{-1} Z_1 S^{-1} Z_2^* D Z_1
                    \\
                    &= I_1 + S^{-1}Z_2^* G^{-1}  \left [\left(D^{-1} + Z_1 S^{-1} Z_2^*\right) D - I_2 -   Z_1 S^{-1} Z_2^* D \right] Z_1
                    \\
                    &= I_1 + S^{-1}Z_2^* G^{-1} \left [I_2 + Z_1 S^{-1} Z_2^* D - I_2 -  Z_1 S^{-1} Z_2^* D\right] Z_1
                    \\
                    & = I_1
                \end{align*}
                and therefore, $B^{-1} = C B B^{-1} = C$ holds.
            \end{proof}

            The sesquilinear form $a_\alpha$ is continuous outside of the singularities, which can be proven analogously to \cite[Lemma~6]{Konschin19a}
            \begin{lemma}
                If Assumption~\ref{assumption_absorptionMaxwell} holds, then the solution operator for the sesqui\-linear form $a_\alpha$ is continuous in $\overline{I}_{} \setminus \mathcal{A}$.
            \end{lemma}
            
            Now, we  consider the convergence of the solution operator $L_\alpha$, if $\alpha$ approaches a singularity $\hat{\alpha} \in \mathcal{A}$.
            For that, we fix $\hat{\alpha} \in \mathcal{A}$ and consider the finite set
            \begin{equation}\label{eq_Jalpha}
                J_{\hat{\alpha}}
                := \left\{ j \in \Z^2 : |\hat{\alpha} +  j | = k \right\}
                .
            \end{equation}
            We decompose the boundary operator $N_\alpha$ into two parts, one part is a finite sum with all the singularities and the second part does not include any singularity. We define for  $\alpha \in \overline{I}_{}\setminus \mathcal{A}$,  $u_\alpha \in \tilde{X}_\alpha$, and for $\tx \in \Gamma_0^R$ the operators
            \begin{align*}
                N_\alpha(u_{\alpha,T}) (\tx)
                &=   \left[ \frac{1}{2 \pi}  \sum_{j \not \in J_{\hat{\alpha}}} \frac{-\ii\alpha_j}{\sqrt{k^2 - | \alpha_j|^2}} \left( \alpha_j \cdot \widehat{( u_{\alpha, T})}_j\right) e^{-\ii \alpha_j \cdot \tx}
                \right] \\
                &\quad\quad\quad
                + \left[  \frac{1}{2 \pi}  \sum_{j \in J_{\hat{\alpha}}} \frac{-\ii\alpha_j}{\sqrt{k^2 - | \alpha_j|^2}} \left( \alpha_j \cdot \widehat{( u_{\alpha, T})}_j \right) e^{-\ii \alpha_j \cdot \tx} \right]
                \\
                &=: 
                \tilde{N}_\alpha(u_{\alpha,T}) (\tx) 
                + \hat{N}_\alpha(u_{\alpha,T}) (\tx)
                ,
            \end{align*}
            where $\tilde{N}_\alpha(u_{T, \alpha})$ is well-defined for all $\alpha$ in a neighborhood $U(\hat{\alpha}) \cap I_{}$ of $\hat{\alpha}$.
            If we set the linear and continuous functional
            \[
                l_{\alpha_j} : H_\alpha(\curl; \Omega_0^R) \to \CC
                ,
                \quad 
                u_\alpha \mapsto (\alpha +  j) \cdot \widehat{(u_{\alpha,T})}_{j}
                ,
            \]
            then we  can rewrite the operator $\hat{N}_\alpha(u_{\alpha,T})$ as
            \begin{equation*}
                \int_{\Gamma_0^{R}} \hat{N}_\alpha(u_{\alpha,T})\, \overline{ v }_{\alpha,T}
                \dd  S
                =  
                \frac{1}{2 \pi} \sum_{j \in J_{\hat{\alpha}}} \frac{- \ii}{\sqrt{k^2 - | \alpha_j|^2}} \, l_{\alpha_j}( u_{\alpha})\, \overline{l_{\alpha_j} ( v_{\alpha}) }
                .
            \end{equation*}
            We define the sesquilinear form  $s_\alpha$, which includes all the parts of $a_\alpha$ beside the singularities, by 
            \begin{align*}
                s_\alpha(u_\alpha, v_\alpha)
                &:= \int_{\Omega_0^R} \mu_r^{-1} \nabla \times u_\alpha \cdot \nabla \times \overline{v}_\alpha - k^2 \epsilon_r u_\alpha \cdot \overline{v}_\alpha \dd  x
                \\
                \nonumber
                &\quad \quad\quad
                +   \int_{\Gamma_0^{R}}  \tilde{N}_\alpha(u_{\alpha,T}) \cdot \overline{v}_{\alpha,T} - T_\alpha(u_{\alpha,T} )  \cdot  \overline{v}_{\alpha,T}   \dd  S
            \end{align*}
            for all $u_\alpha$, $v_\alpha \in \tilde{X}_\alpha$. Hence, the problem~\eqref{eq_probWirklichPunkt}  can be written as 
            \begin{equation*}
                a_\alpha(u_\alpha, v_\alpha)
                = s_\alpha(u_\alpha, v_\alpha)
                - \frac{1}{2 \pi} \sum_{j \in J_{\hat{\alpha}}} \frac{\ii}{\sqrt{k^2 - | \alpha_j|^2}} \, l_{\alpha_j}( u_{\alpha}) \,\overline{l_{\alpha_j} ( v_{\alpha}) }  
                =f_\alpha(v_\alpha)
                .
            \end{equation*}
            Applying the theorem of Riesz, we reformulate the problem into  a operator equation of the form 
            \begin{align}\label{eq_Operatordarstellung}
                S_\alpha u_\alpha - \frac{1}{2 \pi} \sum_{j \in J_{\hat{\alpha}}} \frac{\ii}{\sqrt{k^2 - | \alpha_j|^2}} \,  ( u_{\alpha}, z_{\alpha_j})\, z_{\alpha_j}  
                = y_\alpha
                ,
            \end{align}
            where
            $
            s_\alpha(u_\alpha, v_\alpha)
            = (S_\alpha u_\alpha, v_\alpha)_{H(\curl; \Omega_0^R)}
            $,
            $
            l_{\alpha_j}(v_\alpha) 
            = (v_\alpha, z_{\alpha_j})_{H(\curl; \Omega_0^R)}
            ,
            $
            and the right-hand side satisfies
            $
            f_\alpha(v_\alpha)
            = (y_\alpha, v_\alpha)_{H(\curl; \Omega_0^R)}
            .
            $

            \begin{lemma}
                If Assumption~\ref{assumption_absorptionMaxwell} holds, then the operator $S_\alpha$ is continuously invertible and the map $\alpha \mapsto S_\alpha^{-1}$ is unified continuous in an open ball around $\hat{\alpha}$.
            \end{lemma}
            \begin{proof}
                The continuous invertibility of $S_\alpha$ can be shown analogously to the invertibility of the differential operator corresponding to $a_\alpha$. Since the operator $S_\alpha$ is well-defined in every $\alpha \in \overline{I}$ and continuous on $\overline{I}$, the Neumann series argument implies that the map $\alpha \mapsto S_\alpha^{-1}$ is unified continuous (compare \cite[Lemma~6]{Konschin19a}).
            \end{proof}
            
            To simplify the notation in the following argumentation, we renumber the $|J_{\hat{\alpha}}|$ elements $\{z_{\alpha_j}\}_{j \in J_{\hat{\alpha}}}$ as $\{z_m\}_{m=1,\ldots,|J_{\hat{\alpha}}|}$ and call the corresponding $\alpha_j$ as $\alpha^m$.
            Further, we define the operator $Z^*_{\alpha} \colon \tilde{X}_\alpha \to \CC^{|J_{\hat{\alpha}}|}$ and his adjoint operator $Z_{\alpha} \colon \CC^{|J_{\hat{\alpha}}|} \to  \tilde{X}_\alpha $  by
            \begin{equation*}
                Z^*_{\alpha} \colon v \mapsto \left\{(v, z_m) e_m \right\}_{m=1,\ldots,|J_{\hat{\alpha}}|}
                \quad\text{and}\quad
                Z_{\alpha} \colon x \mapsto \sum_{m=1}^{|J_{\hat{\alpha}}|} (x, e_m)\, z_m
                .
            \end{equation*}
            Since every $z_m$, $m=1,\ldots,|J_{\hat{\alpha}}|$, corresponds to a different $j \in J_{\hat{\alpha}}$, the set  $\{ z_m \}_{m=1}^{|J_{\hat{\alpha}}|}$ is linearly independent and $\mathcal{N}(Z_{\alpha} ) = \mathcal{R}(Z_{\alpha} ^*)^\perp = \{0 \}$ holds.
            Moreover, we define the diagonal matrix 
            \[
                D_\alpha \colon \CC^{|J_{\hat{\alpha}}|} \to \CC^{|J_{\hat{\alpha}}|}
                ,\ 
                e_m \mapsto \frac{-\ii}{2 \pi} \frac{1}{\sqrt{k^2 - | \alpha^m|^2}} \, e_m
                \quad\text{for all }
                m = 1,\ldots,|J_{\hat{\alpha}}|
                ,
            \]
            such that we can write the operator equation~\eqref{eq_Operatordarstellung}  as
            \[
                \left(S_\alpha + Z_\alpha D_\alpha Z_\alpha^* \right) u_\alpha
                = y_\alpha
                .
            \]
            
            \begin{lemma}\label{lemma_neq0} 
                If the Assumption~\ref{assumption_absorptionMaxwell} holds, then 
                the operator $Z^*_{\alpha} S_\alpha^{-1}Z_{\alpha} \colon \CC^{|J_{\hat{\alpha}}|} \to \CC^{|J_{\hat{\alpha}}|}$ is continuously invertible in a neighborhood of $\hat{\alpha}$.
            \end{lemma}
            \begin{proof}
                Because of the assumptions on the parameters
                $\varepsilon_r$ and $\mu_r$ the operator $S_\alpha$ is continuously invertible. 
                Since we are in the setting of a finite dimensional space $\mathcal{L}(\CC^{|J_{\hat{\alpha}}|})$, the invertibility follows by the injectivity of $Z^*_{\alpha} S_\alpha^{-1}Z_{\alpha}$.
                
                Let's assume the operator is not one-to-one, then there exists a vector $v$ in the kernel $\mathcal{N}(Z_\alpha^* S^{-1}_\alpha Z_\alpha) \setminus \{0 \}$. Since it holds $\mathcal{N}(Z_{\alpha} ) = \{0 \}$, we derive $w := S^{-1}_\alpha Z_\alpha v \neq 0$. On the other hand, $w$ solves
                \begin{align*}
                    0
                    &= \left(Z_\alpha^* S^{-1}_\alpha Z_\alpha v, v \right)_{\CC^{|J_{\hat{\alpha}}|}}
                    = \left(S^{-1}_\alpha Z_\alpha v, Z_\alpha v\right)_{\CC^{|J_{\hat{\alpha}}|}}
                    = (w, S_\alpha w)_{\CC^{|J_{\hat{\alpha}}|}}
                    \\
                    &= \overline{(S_\alpha w, w)}_{\CC^{|J_{\hat{\alpha}}|}}
                    = \overline{s_\alpha(w, w)}
                    ,
                \end{align*}
                wherefrom we can show analogously to Lemma~\ref{lemma_perRWPLoesung} that $w = 0$ holds. Thus, we have a contradiction.
            \end{proof}
            
            \begin{theorem}\label{satz_darstellungDerInversen}
                If Assumption~\ref{assumption_absorptionMaxwell} holds, then in the neighborhood $U(\hat{\alpha}) \subseteq I \setminus \mathcal{A}$ of $\hat{\alpha}$ the equation
                \begin{align}\label{eq_SalphaPlusRest}
                    \left(S_\alpha + Z_\alpha D_\alpha Z_\alpha^* \right) u_\alpha
                    = y_\alpha
                \end{align}
                is uniquely solvable and the solution operator can be written as
                \begin{align}\label{eq_smwFormel}
                    S_\alpha^{-1} - S_\alpha^{-1} Z_{\alpha} \left(D^{-1}_{\alpha} + Z_{\alpha}^* S_\alpha^{-1} Z_{\alpha} \right)^{-1}Z_{\alpha}^* S_\alpha^{-1}
                \end{align}
                in this neighborhood.
            \end{theorem}
            \begin{proof}
                We showed the  invertibility of \eqref{eq_SalphaPlusRest} in Lemma~\ref{lemma_perRWPLoesung}.
                For some $\alpha\in \overline{I}_{} \setminus \mathcal{A}$, which is close to $\hat{\alpha} \in \mathcal{A}$, the invertible matrix $D_\alpha^{-1}$ converges to the zero matrix, such that the Neumann series argument implicates the invertibility of $D_\alpha^{-1} + Z_\alpha^* S^{-1}_\alpha Z_\alpha $.
                Thus, all assumptions of the Sherman-Morrison-Woodbury formula are fulfilled  and the inverse of $S_\alpha + Z_\alpha D_\alpha Z_\alpha^* $ can be written in the form of \eqref{eq_smwFormel} by Theorem \ref{satz_SMWFormel}.
            \end{proof}
            
            Thus, we have all the ingredients to prove Theorem~\ref{satz_pktweise1}.
            
            \begin{proof}[Proof of Theorem~\ref{satz_pktweise1}]
                For every wave number $k>0$ the lines with the singularities consist of finite number of parts of the circle going through $\overline{I}_{}$.
                Therefore, for every $\hat{\alpha} \in \mathcal{A}$ there exists a sequence $\{\alpha_j\}_{j=1}^\infty \subseteq \overline{I}_{} \setminus \mathcal{A}$, which converges to $\hat{\alpha}$.
                Let $U(\hat{\alpha})$  be a small neighborhood around $\hat{\alpha}$. The operators $S_\alpha^{-1}$, $Z_\alpha$, $D_\alpha^{-1}$ and $Z_\alpha^*$ are continuous (or continuously extendable) to $U(\hat{\alpha})$. The Neumann series argument implies that $(D^{-1}_{\alpha} + Z_{\alpha}^* S_\alpha^{-1} Z)^{-1}$ is also continuous on $U(\hat{\alpha})$. Hence, the representation of the sesquilinear form $a_\alpha$ in \eqref{eq_smwFormel}  converges to 
                \begin{equation*}
                    S_\alpha^{-1} - S_\alpha^{-1} Z_{\alpha} \left(D^{-1}_{\alpha} + Z_{\alpha}^* S_\alpha^{-1} Z_{\alpha}\right)^{-1} Z_{\alpha}^* S_\alpha^{-1}
                    \to S_{\hat{\alpha}}^{-1} - S_{\hat{\alpha}}^{-1} Z_{\hat{\alpha}} \left(Z_{\hat{\alpha}}^* S_{\hat{\alpha}}^{-1} Z_{\hat{\alpha}}\right)^{-1}Z_{\hat{\alpha}}^* S_{\hat{\alpha}}^{-1}
                \end{equation*}
                for  $\alpha \to \hat{\alpha}$ w.r.t.\ the operator norm of $\mathcal{L}(\tilde{X}_\alpha)$, since the limit is well-defined by Lemma~\ref{lemma_neq0}.
                
                Thus, the mapping $\alpha \mapsto L_\alpha$, where $L_\alpha$ is the solution operator of the quasi-periodic problem~\eqref{eq_probWirklichPunkt}, is continuously extendable to $\overline{I}_{}$
                and, in particular, unified continuous on $\overline{I}_{}$.
                Therefore, there exists an  $\alpha$ independent constant $ C > 0$, such that $\sup_{\alpha \in \overline{I}_{}} ||L_\alpha|| < C$  holds.
                In consequence, the function $\tilde{u}(\alpha, \cdot) := u_\alpha$ is an element of  $L^2(I_{}; \tilde{X}_\alpha)$ and solves the problem~\eqref{eq_transVarProblem}. It remains to show that $\tilde{u}$ actually lies in $L^2_{\mathrm{w}}(I_{}; \tilde{X}_\alpha)$ and the estimation~\eqref{eq_abschaetzung} holds.
                Since for $\alpha \in I_{} \setminus \mathcal{A}$ the sum $\langle\hat{N}_\alpha(u_{\alpha,T}), u_{\alpha,T} \rangle$ only consists of entries with negative real part and vanishing imaginary part, or, negative imaginary part and vanishing real part, we can estimate
                \begin{align*}
                    &\frac{1}{2}\sum_{j \in J_{\hat{\alpha}}} \left|k^2 - | \alpha_j|^2\right|^{\nicefrac{-1}{2}}  |l_{\alpha_j}( u_{\alpha})|^2  
                    \leq\left|\sum_{j \in J_{\hat{\alpha}}} \frac{\ii}{\sqrt{k^2 - | \alpha_j|^2}} \,  l_{\alpha_j}( u_{\alpha})\, \overline{l_{\alpha_j}( u_{\alpha})}  \right|
                    \\
                    &\leq 2 \pi \,|a_\alpha(u_{\alpha}, u_{\alpha})|
                    +\left|2 \pi\, a_\alpha(u_{\alpha}, u_{\alpha}) + \sum_{j \in J_{\hat{\alpha}}} \frac{\ii }{\sqrt{k^2 - | \alpha_j|^2}} 
                    \,l_{\alpha_j}( u_{\alpha})\, \overline{l_{\alpha_j}( u_{\alpha})}
                    \right|
                    \\
                    &\leq 2 \pi\, ||f_\alpha||_{L^2(\Omega_0^R)} ||u_{\alpha}||_{L^2(\Omega_0^R)} + 2 \pi\, |s_\alpha(u_{\alpha}, u_{\alpha})|
                    \\
                    &\leq C \left(||f_\alpha||^2_{L^2(\Omega_0^R)} + |s_\alpha(u_{\alpha}, u_{\alpha})|\right)
                    .
                \end{align*}
                Moreover, the continuity of the sesquilinear form $s_\alpha$ allows the estimation
                \begin{equation*}
                    |s_\alpha(u_{\alpha}, u_{\alpha})|
                    \leq C\left(C_{\mathrm{trace}} + k^2||\varepsilon_r||_{L^\infty(\Omega_0^R)}+||\mu||^{-1}_{L^\infty(\Omega_0^R)}\right) ||u_{\alpha}||^2_{\tilde{X}_\alpha}
                    \leq C\,||f_{\alpha}||^2_{L^2(\Omega_0^R)}
                    .
                \end{equation*}
                Consequently, we derive the claimed estimation by
                \[
                    \int_{I_{}} \sum_{j \in J_{\hat{\alpha}}} \left|k^2 - | \alpha_j|^2\right|^{\nicefrac{-1}{2}}   |l_{\alpha_j}( u_{\alpha})|^2   \dd  \alpha
                    < \infty
                    ,
                \]
                wherefrom $\tilde{u} \in L^2_{\mathrm{w}}(I_{}; \tilde{X}_\alpha)$ and the estimation~\eqref{eq_abschaetzung} follows.
            \end{proof}

    \section{Existence theory in case of a local perturbed permittivity}\label{Sec5}
        In this section we consider the scattering problem including a local perturbation in the permittivity $\varepsilon_r$. We call ${\varepsilon}^{\mathrm{s}}_r$  the perturbed parameter, and assume that the perturbation $q := {\varepsilon}^{\mathrm{s}}_r - \varepsilon_r $ has the support $\supp(q)$ in $\Omega^R_0$. 
        Moreover, the imaginary part of the perturbed parameter should satisfy   $(\Im {\varepsilon}^{\mathrm{s}}_r) \geq  0$ as assumed in Assumption~\ref{assumption_absorptionMaxwell}.

        The idea is to apply the Fredholm alternative to show the solvability. Therefore, we define the two spaces $Y$ and $Y^{\perp}$ as
        \[
            Y
            := \left\{ u \in X : \divv(\epsilon_r u) = 0 \right\}
            = \JJ_{\R^2}^{-1} \left( L^2(I; Y_\alpha)\right)
        \]
        and
        \[
            Y^{\perp} 
            := \{ u \in X : u = \nabla w,\ w \in W_0 \}
            = \JJ_{\R^2}^{-1} \left(L^2(I; Y^{\perp}_\alpha)\right)
            .
        \]
        Since the Bloch-Floquet transform and the partial derivation can be interchanged, we derive the decomposition $X = Y \oplus Y^{\perp} $.
        \begin{lemma}
            The subspaces $Y$ and $Y^{\perp}$ of $X$ are closed and $X$ can be decomposed in
            $
                X 
                = Y \oplus Y^{\perp}
                .
            $
        \end{lemma}
        
        Thus, we can split every function in $X$ into the sum of a unique function of $Y$ and a unique function of $Y^{\perp}$. The variational problem reduces to:
        
        \emph{For the right-hand side}
        \[
            g(v)
            := \int_{\Omega^R} f \cdot \overline{y}_v \dd  x - a_q(\nabla w_u, v)
            = \int_{\Omega^R} \left(f+ k^2  \epsilon_r \nabla w_u \right) \cdot \overline{v} \dd  x 
        \]
        \emph{we seek $u \in Y $, such that}
        \begin{equation}\label{eq_ProblemnachHelmholtzPert}
            a_q(u, v)
            = g(v)
            \quad
            \text{\emph{holds for all }}
            v \in Y
            .
        \end{equation}
        \begin{theorem}
            If the Assumption~\ref{assumption_absorptionMaxwell} holds, then 
            there exists the unique solution ${\mathbf E} \in X$ for the variational problem~\ref{prob_Maxwell}.
        \end{theorem}
        \begin{proof}
            As we have seen before, we only have to show the unique solvability of the reduced problem~\eqref{eq_ProblemnachHelmholtzPert} by $u \in Y$.
            Because of the regularity of $\epsilon_r \in W^{1, \infty}(\R^3_+)$ and the embedding of  $Y \subseteq H^1(\Omega^R)$, the sesquilinear form $l : X \times X \to \mathbb{C}$,
            \[
                l(u, v) 
                := \int_{\Omega^R}  -k^2 q  u \cdot \overline{v}  \dd  x 
                ,
            \]
            is a compact perturbation of the differential operator, since $ u \in H^1(\Omega^R)$, $\supp(q) \subseteq \Omega_0^R$ and $H^1(\Omega_0^R)$ is compactly embedded into $L^2(\Omega^R_0)$ (see \cite[Theorem~3.27]{McLea2000}).
            Therefore, it remains to show the uniqueness and derive the existence by the Fredholm alternative.
            
            Let $u$ be the solution for the right-hand side $g = 0$. Because of the condition $(\Im {\varepsilon}^{\mathrm{s}}_r) \geq 0$, we can estimate the sesquilinear form by
            $
                0
                \leq   \int_{\Omega^R} - (\Im {\varepsilon}^{\mathrm{s}}_r ) |u|^2  \dd x
                \leq 0
                .
            $
            We assumed that $(\Im\epsilon_r) > 0$ holds on an open ball of $\Omega^R$. Hence, the function ${u}$ vanishes on this set and the unique continuation property in Proposition~\ref{proposition_uniqueContienuationMaxwell} implicates that $w$ has to vanish everywhere.
        \end{proof}

    \section{Regularity of the transformed solution w.r.t.\ the quasi-periodicity}\label{Sec6}
        In this section we consider the regularity of the transformed solution to the problem~\ref{prob_Maxwell} w.r.t.\ the quasi-periodicity.
        At first, we consider the regularity in the case of unperturbed periodic parameters.
        
        \begin{theorem}\label{satz_regUngestoert}
            Let $f_\alpha \in L^2(\Omega_0^R)$ be analytical in  $\alpha \in \overline{I_{}}$. Then the solution ${\mathbf E}_\alpha\in \tilde{X}_\alpha$ of the (unperturbed) quasi-periodic variational problem~\eqref{eq_probWirklichPunkt} is continuous in $\alpha \in \R^2$ and analytical in $\alpha \in \overline{I_{}} \setminus \mathcal{A}$.
            For any $\alpha \in \overline{I_{}} \setminus \mathcal{A}$ which is near a singularity $\hat{\alpha} \in \mathcal{A}$, there exist functions  ${\mathbf E}_\alpha^1$ and ${\mathbf E}_{\alpha_j}^2 \in \tilde{X}_\alpha$, $j \in J_{\hat{\alpha}}$, which are analytical in $\alpha \in \R^2$, such that
            \begin{equation}\label{eq_zerlegung}
                {\mathbf E}_\alpha
                = {\mathbf E}_\alpha^1 + \sum_{j \in J_{\hat{\alpha}}} \sqrt{k^2-|\alpha_j|^2}\, {\mathbf E}_{\alpha_j}^2
                .
            \end{equation}
        \end{theorem}
        \begin{proof}
            For $\alpha \in \overline{I_{}} \setminus \mathcal{A}$ the differential operator is analytical w.r.t.\ $\alpha$, and hence, the solution is analytical there.
            Therefore, we only have to show how the solution behaves near a singularity.
            If $\alpha$ is near a singularity $\hat{\alpha} \in \mathcal{A}$, then the solution operator can be decomposed into
            \[
                S_\alpha^{-1} - S_\alpha^{-1} Z_{\alpha} \left(D^{-1}_{\alpha} + Z_{\alpha}^* S_\alpha^{-1} Z_{\alpha}\right)^{-1}Z_{\alpha}^* S_\alpha^{-1}
                ,
            \]
            which follows by Theorem~\ref{satz_darstellungDerInversen}.
            The operators  $Z_{\alpha}$ and $Z_{\alpha}^*$ are analytical in  $\alpha \in \overline{I_{}}$ and we can show the analyticity of $S_\alpha^{-1}$ analogously to \cite[Lemma~6]{Konschin19a}.
            Hence, the linear operator $G_\alpha:=(Z_\alpha^* S_\alpha^{-1} Z_\alpha)^{-1}$ is well-defined by Lemma~\ref{lemma_neq0} and the Neumann series argument implies that $G_\alpha$ is analytical.
            If $\alpha$ approaches $\hat{\alpha}$, then the matrix $D_\alpha^{-1}$ convergences with an order of one half to the zero matrix and applying the Neumann series argument and we derive  the equation
            \begin{align*}  
                \left(D_\alpha^{-1} + Z_\alpha^* S_\alpha^{-1} Z_\alpha \right)^{-1} v
                & = \left( G_\alpha D_\alpha^{-1} + I_{|J_{\hat{\alpha}}|} \right)^{-1}  G_\alpha v
                =- \sum_{l=0}^{\infty}  \left( G_\alpha D_\alpha^{-1} \right)^{l}  G_\alpha  v
                \\
                &=: U^0_\alpha v + \sum_{m=1}^{|J_{\hat{\alpha}}|}  \sqrt{k^2-|\alpha^m|^2} \,U^m_\alpha v
            \end{align*}
            for analytical dependent operators   $U^m_\alpha$, $m = 0,\ldots,|J_{\hat{\alpha}}|$ and for a vector $v \in \CC^{|J_{\hat{\alpha}}|}$.
            Therefore, the solution can also be written as
            \[
                {\mathbf E}_\alpha
                = S_\alpha^{-1}  f_\alpha - S_\alpha^{-1} Z_{\alpha} \left(D^{-1}_{\alpha} + Z_{\alpha}^* S_\alpha^{-1} Z_{\alpha} \right)^{-1}Z_{\alpha}^* S_\alpha^{-1} f_\alpha
                .
            \]
        \end{proof}
        
        In the next theorem, we consider the scattering problem with a locally perturbed permittivity.
        
        \begin{theorem}
            Let $\JJ_{\R^2} {\mathbf E} \in L^2_{\mathrm{w}}(I_{}; \tilde{X}_\alpha)$ be the Bloch-Floquet transformed solution to the (locally perturbed) variational problem~\ref{prob_Maxwell} with the right-hand side $f \in L^2(\Omega^R)$, such that $\JJ_{\R^2} f$ is analytical in $\alpha \in \overline{I_{}}$.
            Then $\JJ_{\R^2} {\mathbf E}$  is continuous in  $\alpha \in \R^2$ and analytical in $\alpha \in \overline{I_{}} \setminus \mathcal{A}$. For $\alpha \in \overline{I_{}} \setminus \mathcal{A}$ near some singularity  $\hat{\alpha} \in \mathcal{A}$ there exist functions ${\mathbf E}_\alpha^1$ and ${\mathbf E}_{\alpha_j}^2 \in \tilde{X}_\alpha$, $j \in J_{\hat{\alpha}}$, which are analytically dependent on $\alpha$, such that  $\JJ_{\R^2} {\mathbf E}$ can be written as
            \begin{equation*}
                \JJ_{\R^2} {\mathbf E}
                = {\mathbf E}_\alpha^1 + \sum_{j \in J_{\hat{\alpha}}} \sqrt{k^2-|\alpha_j|^2}\, {\mathbf E}_{\alpha_j}^2
            .
            \end{equation*}
        \end{theorem}
        \begin{proof}
            The operator $K_q \colon \tilde{X} = \JJ_{\R^2} X \to L^2(I_{} \times \Omega_0^R)$, $\tilde{u} \mapsto q \JJ_{\R^2}^{-1} \tilde{u}$, maps functions of $\tilde{X}$ to functions which are constant in $\alpha$, and, in particular, analytical in $\alpha$. Let $A$ be the Riesz representation of the invertible unperturbed Bloch-Floquet transformed differential operator, $\tilde{K}_q \in \mathcal{L}(\tilde{X})$ the Riesz representation of $K_q$ and $\tilde{f}$ the Riesz representation of $\JJ_{\R^2} f$. 
            Then the solution $\JJ_{\R^2} {\mathbf E} \in \tilde{X} = L^2_{\mathrm{w}}(I_{}; H_\alpha(\curl; \Omega_0^R))$  to the problem~\ref{prob_Maxwell} satisfies the equation
            \[
                \JJ_{\R^2} {\mathbf E}
                = A^{-1} \tilde{f} - A^{-1}  \tilde{K}_q {\mathbf E}
                \quad\text{in }
                \tilde{X}
                .
            \]
            Since $\tilde{f}$ as well as $\tilde{K}_q w$ are analytical in $\alpha$, Theorem~\ref{satz_regUngestoert} implicates that $\JJ_{\R^2} {\mathbf E}$ can be written in the claimed representation.
        \end{proof}
        
        \begin{remark}
            It is sufficient to have a right-hand side which fulfills an analogous decomposition as \eqref{eq_zerlegung} to have the same decomposition to the solution.
        \end{remark}

	\subsubsection*{Acknowledgment}
        We thank Prof. Dr. Andreas Kirsch for the valuable suggestions, which improved this work.
	
\printbibliography
\end{document}